\documentclass[final,10pt]{aimdyn-rep}

\renewcommand\reportyear{2017}  
\renewcommand\reportmonthcreated{08}	
\renewcommand\reportnumber{004} 
\renewcommand\thisversion{1}  
\renewcommand\revisiondate{\today}
\renewcommand\thiscontract{DARPA HR0011-16-C-0116; ARL W911NF-16-2-0160}  

\usepackage{charter}   

\usepackage[
    style=numeric,  
    citestyle = numeric, 	
    giveninits=true,
    backend=bibtex,  
    natbib=true,
    url=false, 
    useprefix=false,
    sorting=none]{biblatex}
\usepackage[margin=1in]{geometry}
\usepackage[latin1]{inputenc}
\usepackage{color}
\usepackage{algorithmic}
\usepackage{algorithm}
\usepackage{wasysym}
\usepackage{url}
\usepackage{float}
\usepackage{relsize}
\usepackage[makeroom]{cancel}
\usepackage{marginnote}
\usepackage{eucal}

\usepackage{mathrsfs}  
\usepackage{mathtools}
\usepackage{bm} 	
\usepackage{bbm} 	

\newcommand{\mc}{\mathcal}
\newcommand{\msf}{\mathsf}

\newcommand{\C}{ \ensuremath{\mathbb{C}}}  

\DeclarePairedDelimiterX\set[1]{\lbrace}{\rbrace}{  #1 }  

\DeclarePairedDelimiterX\norm[1]{\lVert}{\rVert}{#1}  			
\DeclarePairedDelimiterX\inner[2]{(}{)}{#1 \,,\, #2}  	

\DeclarePairedDelimiterX\abs[1]{\lvert}{\rvert}{#1} 

\DeclareMathOperator{\spn}{span}

\newcommand{\red}[1]{{\color{red} #1}}

\newcommand{\koop}{\mathfrak{K}}

\newcommand{\A}{\mathbb{A}}

\newcommand{\ab}{\mathbf{a}}
\newcommand{\X}{\mathbf{X}}
\newcommand{\Y}{\mathbf{Y}}
\newcommand{\F}{\mathbf{F}}

\newcommand{\f}{\mathbf{f}}
\newcommand{\x}{\mathbf{x}}
\newcommand{\y}{\mathbf{y}}

\newcommand{\ii}{\mathfrak{i}}

\newcommand{\roff}{{\boldsymbol\varepsilon}}
\DeclareMathOperator*{\argmin}{arg\,min}

\newcommand\restr[2]{{
		\left.\kern-\nulldelimiterspace 
		#1 
		\vphantom{\big|} 
		\right|_{#2} 
	}}

\begin{document}


\title[Enhanced DMD]{Data driven modal decompositions: analysis and enhancements}

\author[Drma\v{c}, Mezi\'{c}, Mohr]{Zlatko Drma\v{c}\affil{1}\comma\corrauth, Igor Mezi\'{c}\affil{2}, Ryan Mohr\affil{3}}

\address{\affilnum{1}\ Faculty of Science, Department of Mathematics, University of Zagreb, Croatia.\\
\affilnum{2}\ Department of Mechanical Engineering and Mathematics, University of California, Santa Barbara, CA, 93106, USA; AIMdyn, Inc., Santa Barbara, CA 93101, USA.\\
\affilnum{3}\ AIMdyn, Inc., Santa Barbara, CA 93101, USA; Department of Mechanical Engineering, University of California, Santa Barbara, CA, 93106, USA.}

\emails{{\tt drmac@math.hr} (Z.\ Drma\v{c}), {\tt mezic@engineering.ucsb.edu} (I.\ Mezi\'{c}), {\tt mohrr@aimdyn.com} (R.\ Mohr)}
%

\begin{abstract}
The Dynamic Mode Decomposition (DMD) is a tool of trade in computational data driven analysis of fluid flows. 
More generally,  it is a computational device for Koopman spectral analysis of nonlinear dynamical systems, with a plethora of applications in applied sciences and engineering. 
Its exceptional performance triggered developments  of several modifications that  make the DMD an attractive method in data driven framework.
This work offers further improvements of the DMD to make it more reliable, and to enhance its functionality.  In particular, data driven formula for the residuals allows selection 
of the Ritz pairs, thus providing more precise spectral information of the underlying Koopman operator, and  the well-known technique of refining the Ritz vectors is adapted to data driven scenarios.
Further, the DMD is formulated in a more general setting of weighted inner product spaces, and the consequences for numerical computation are discussed in detail.
Numerical experiments are used to illustrate the advantages of the proposed method, designated as DDMD\_RRR (Refined Rayleigh Ritz Data Driven Modal Decomposition). 
\end{abstract}



\maketitle

\renewcommand{\thefootnote}{\arabic{footnote}}
\setcounter{footnote}{0}


\section{Introduction}
Dynamic Mode Decomposition (DMD) has become a major tool in the data-driven analysis of complex dynamical systems. DMD was first introduced in 2008 by P. Schmid \cite{Schmid:2008wv} for the study of fluid flows where it was conceptualized as an algorithm to decompose the flow field into component fluid structures, called ``dynamic modes'' or ``DMD modes'', that described the evolution of the flow. The method asserts the existence of a \emph{linear} operator that maps a collection of snapshots of the fluid flow forward one step in time \cite{Schmid:2008wv}. For a nonlinear evolution operator, such as the one generated by the Navier-Stokes equations, the proposed linear operator is equivalent to a linear tangent space approximation \cite{Schmid:2010ba}. The DMD modes and their temporal behavior are given by the spectral analysis of the linear operator, which is constructed from data since it is assumed that direct access to it is not available. Rowley et al.\ \cite{Rowley:2009ez} gave the method theoretical underpinnings by connecting it to the spectral analysis of the Koopman operator --- a linear operator that can be associated with any \emph{nonlinear} dynamical system --- which evolves observables of that system forward in time. The algorithm was cast as a Krylov subspace method in which the operator was represented as a companion matrix in the Krylov basis formed from the data snapshots.

Many variants of the basic algorithm have been introduced since then (see, for example, \cite{Chen:2012jh,Tu2014391,Williams:2015kh,Hemati:2014jm,Jovanovic:2012wy,Jovanovic:2014ft}) all purporting to more accurately, robustly, or efficiently compute the eigenvalues and modes under various assumptions on the data. However, deep numerical analyses giving some certificate of accuracy for these algorithms have been absent. This is especially troubling as the DMD method, in all of its various guises, has enjoyed large scale deployment in fields such as fluid dynamics (see \cite{Mezic:2013} and the references therein) where it is often taken as an ``off-the-shelf'' algorithm whose results are implicitly trusted. This is contrasted with the subset of practitioners who recognize that the method often produces spurious or inaccurate eigenvalues that are not associated with spectrum of the operator generating the data. This can even be true in the simplest case where the data snapshots are produced by powers of a matrix applied to an initial vector --  the standard Schmid-type DMD method can fail to accurately capture the spectrum of the matrix, even if the supplied data is rich in spectral information.

The detection of these spurious or inaccurate eigenvalues has been approached in an ad hoc manner. Eigenvalues are often ranked in decreasing importance by the $L^2$-norm (energy) of their associated mode and are deemed non-essential if the norm is sufficiently small. Recently, D. Giannakis has proposed using a different measure that imposes a penalty for the eigenvalues based on their mode's ``roughness'' {via their Dirichlet energy} \cite{Giannakis:2015wr}. This modification captures the physical reasoning that real systems are more likely to produce smooth modes, which is a conjecture that itself must be justified. 

Focusing on the magnitude of the energy, however, can lead to discarding physically relevant dynamics, especially if the high energy of the mode is an artifact of the units the data is reported in. For example, snapshots can be formed from data acquired via several different sensors, with each sensor reporting information in different units. From a scientist's perspective, there is no difference in reporting, say the power consumption of a system, in watts or milliwatts; both numbers represent the same physical quantity. Numerically, however, there can be a large difference. 
The situation is exacerbated further by data which contains measurements of quantities with fundamentally different physical nature.

Despite these concerns, DMD methods have demonstrated exceptional performance in many applications. However, this often requires deep, domain-specific knowledge to determine the reliability of the algorithm's output. The question of when it fails and how badly is still open. If the outputs of the algorithm are not known to be spurious or real, the inferences based on these outputs cannot be known to be reliable. Therefore, DMD should be analyzed in depth so that there are guarantees on the accuracy and reliability of the algorithm. Furthermore, this analysis should be divorced from domain-specific knowledge of the current application. This will not only reassure the algorithm's fitness for further nontrivial applications, moving it toward a true ``off-the-shelf'' method that a non-expert can apply to their particular problem, but also allow modifications that will improve its numerical reliability and robustness.

\subsection{Contributions and overview}
In this work, we excogitate ways to address the aforementioned issues with several modifications and enhancements of the DMD.
In \S \ref{zd:S=Intro-K+RR}, we set the stage and briefly review Krylov subspaces with the corresponding decomposition, and the Rayleigh-Ritz procedure for extracting spectral information from a given Krylov subspace. We briefly discuss how the Krylov subspaces naturally appear in spectral approximations of the Koopman operator, and we review the DMD algorithm. In \S \ref{S=New-DMD-RRR}, we first show how the DMD algorithm can be equipped with residual estimate that can be used to assess the quality of each particular Ritz pair. Further, we show how to apply the well known Ritz vector refinement technique to the DMD data driven setting, and we discuss the importance of data scaling. All these modifications are integrated in \S \ref{SS=DMD-RRR} where we propose a new version of the DMD, designated as DDMD\_RRR (Refined Rauleigh-Ritz Data Driven Modal Decomposition).  In \S\ref{zd:S=Examples} we provide numerical examples that show the benefits of our modifications, and we discuss the fine details of software implementation.  In \S \ref{zd:S=Exact-DMD} we use the Exact DMD \cite{Tu-DMD-Theory-Appl} to show that our modifications apply to other versions of DMD.
In \S\ref{zd:S=Compressed-DMD}, we provide a compressed form of the new DDMD\_RRR, designed to improve the computational efficiency in case of extremely large dimensions. A matrix-root-free modification of the Forward-Backward DMD \cite{Dawson2016} is presented in \S\ref{S=FB-DMD}.
The column scaling used in the new DDMD implementation in \S\ref{SS=DMD-RRR} is just a particular case of a more general weighting scheme that we address in \S\ref{zd:S=General-weighted}. Using the concept of the generalized SVD introduced by Van Loan \cite{van-Loan-GSVD}, we define weighted DDMD with the Hilbert space structures in the spaces of the snapshots and the observables (spatial and temporal) given by two positive definite matrices.

\section{DMD as data driven Krylov+Rayleigh-Ritz procedure}\label{zd:S=Intro-K+RR}
In the framework of dynamic mode decomposition and analysis, we are given e.g. the flow field data\footnote{In some applications the data can be complex.} $\f_1, \ldots, \f_{m+1}\in\mathbb{C}^n$, under the assumption that it has been generated by an unknown linear operator $\A$ such that $\f_{i+1}=\A \f_i$, $i=1,2,\ldots$. 
We can think of $\A$ as a discretization of the underlying physics that drives the measured $\f_i$'s.
In a pure data driven setting we have no access to $\A$. Instead, the $\f_i$'s are the results of measurements, e.g. computed from pixel values from a high speed camera recorded video, see e.g. \cite[\S 3.1]{Schmid2011}. 
 No other information on the action of $\A$ is available. 

In another scenario of data acquisition, $\A$ represents PDE/ODE solver (software toolbox) that generates solution in high resolution, with given initial condition $\f_1$. In such a framework, the discrete time evolution $\f_{i+1}=\A \f_i$ can be stopped at some (time) index $i$ and then restarted with new initial condition.
In both scenarios, $n$ is expected to be large, say $n>10^4$, and the number $m$ of snapshots is typically much smaller. 
The  goal is to extract useful spectral information on $\A$, based solely on these measurements and/or numerical simulation data. 

\subsection{{Connection with the Koopman operator}}\label{SS=DMD-Koopman-connection} 
The seemingly simple sequence of the $\f_i$'s, the result of the power method applied to $\A$, can be interpreted as a discretization of power iterations applied to a linearization of
complex nonlinear dynamics.
In analyzing a nonlinear dynamical system $T : M \to M$ there is an associated infinite-dimensional linear operator $\mc U : \mc H \to \mc H$ defined by the composition operation $\mc U \psi := \psi \circ T$, where $\mc H$ is a Hilbert space of functions on $M$ closed under composition with $T$. 
The spectral properties of this so-called Koopman operator are useful in the analysis, prediction, and control of the underlying nonlinear dynamical system \cite{MezicandBanaszuk:2004,Mezic:2005}. 

There are two essentially different types of appoximations of the Koopman operator that DMD techniques provide \cite{Mezic:201708.003}. The first one is related to the methodology introduced in \cite{Rowley:2009ez}, and is interpreted in \cite{Mezic:201708.003} as follows. Let $\mathcal{S}=\{\x_1,\ldots,\x_m\}$ be an invariant set for $T$. Consider the space  ${\mathcal{C}}|_\mathcal{S}$, of continuous functions in $\mc H$ restricted to $\mathcal{S}$.  This is an $m$-dimensional vector space. The restriction of the Koopman operator to ${\mathcal{C}}|_\mathcal{S}$, $\mc U|_\mathcal{S}$ is then a finite-dimensional linear operator that can be represented in a basis by an $m\times m$ matrix. An explicit example is given when $\x_j,j=1,\ldots, m$ represent successive points on a periodic trajectory, and the resulting matrix representation in the standard basis is the  $m\times m$ cyclic permutation matrix
\begin{equation}
{\displaystyle P={\left(\begin{smallmatrix}0&0&\ldots &0&1\\1&0&\ldots &0&0\\0&\ddots &\ddots &\vdots &\vdots \\\vdots &\ddots &\ddots &0&0\\0&\ldots &0&1&0\end{smallmatrix}\right)}}. 
\end{equation}

If $\mathcal{S}$ is not an invariant set, an $m\times m$ approximation of the reduced Koopman operator can still be provided. Namely, if we know $m$ independent functions' restrictions $(f_j)|_\mathcal{S}$, $j=1,\ldots,m$ in  ${\mathcal{C}}|_\mathcal{S},$ and we also know $f_j(T\x_k)$, $j,k\in\{1,\ldots,m\}$, we can provide a matrix representation of $\mc U|_\mathcal{S}$. However, while in the case  where $\mathcal{S}$ is an invariant set, the iterate of any function in ${\mathcal{C}}|_\mathcal{S}$ can be obtained in terms of the iterate of $m$ independent functions, for the case when $\mathcal{S}$ is not invariant this is not necessarily so. Namely, the fact that $T$ is not invariant  means that functions in ${\mathcal{C}}|_\mathcal{S}$ do not necessarily experience linear dynamics under $\mc U|_\mathcal{S}$. However, one can take $n$ observables $f_j$, $j=1,\ldots,n$, where $n>m$, and approximate the nonlinear dynamics using linear regression on $\f(\x)\equiv (\f(\x_1),\ldots,\f(\x_m)),$ where $\f(\cdot)=(f_1(\cdot),\ldots,f_n(\cdot))^{\msf T}$ -- i.e by finding an $m\times m$ matrix $C$ that gives the best approximation of the data in the Frobenius norm, 
\begin{equation}
C =\underset{B\in\C^{m\times m}}\argmin ||\f(T\x)- \f(\x) {B}||_F \equiv \underset{B\in\C^{m\times m}}\argmin \| (f_j(T\x_k))_{j,k=1,1}^{n,m} -  (f_j(\x_k))_{j,k=1,1}^{n,m} {B}\|_F.
\end{equation}
{Under certain conditions this approximation converges weakly to the Koopman operator on an invariant set  that the $\x_j$'s are densely distributed on, see \cite{Mezic:201708.003}.}
 
DMD algorithms and the spectral analysis of the Koopman operator can also be connected by considering finite sections of the matrix associated with the operator \cite{Mezic:201708.003}. Let $( \phi_1, \phi_2, \dots ) \subset \mc H$ be a (not necessarily orthogonal) basis for $\mc H$ and $(\hat\phi_1, \hat\phi_2, \dots ) \subset \mc H$ the dual basis satisfying {$\inner{\phi_i}{\hat\phi_j} = \delta_{ij}$}.\footnote{{Following standard math notation (as opposed to physics notation), our inner product is \emph{linear in the first variable and conjugate linear in the second}.}} Let $\mc H_n = \spn(\phi_1,\dots, \phi_n)$ and $P_n : \mc H \to \mc H_n$ be the orthogonal projection onto $\mc H_n$. We consider a compression of the operator $\mc U_n := P_n \restr{\mc U}{\mc H_n} : \mc H_n \to \mc H_n$ and find its matrix representation $\A \in \C^{n\times n}$ in the basis $(\phi_1,\dots, \phi_n)$. We first note that $P_n \equiv \Phi_n \hat\Phi_n$ where $\Phi_n : \C^{n} \to \mc H_n$ and $\hat\Phi_n : \mc H \to \C^n$ are given, respectively, by
\begin{align}
&\Phi_{n}((c_1,\dots,c_n)^{\msf T}) = \sum_{k=1}^{n} c_k \phi_k, &  
&{\hat\Phi_n \psi = (\inner{\psi}{\hat\phi_1}, \dots, \inner{\psi}{\hat\phi_n})^{\msf T}}. 
\end{align}
The matrix $\A$ has elements defined as {$\A_{ij} = \inner{\mc U\phi_j}{\hat\phi_i}$} for $1\leq i,j \leq n$ and it can be checked that $\mc U_n = \Phi_n \A \hat\Phi_n$. Since $\hat\Phi_n \Phi_n = I_{\C^n}$, we have the identity $\mc U_{n}^i = \Phi_n \A^i \hat\Phi_n$ for all $i\geq 0$. Now, fixing a function $\psi \in \mc H_n$ and evolving it with the compression $\mc U_n$ gives $\mc U_n^i \psi = \Phi_n \A^i \hat\Phi_n \psi$ for $i\geq 0$. If we define $\f_1 := \hat\Phi_n \psi$ and $\f_{i+1} := \hat\Phi_n \mc U_n^i \psi$ for $i\geq 0$, then we have that $\f_{i+1} = \A^i \f_1$ from the identity $\hat\Phi_n\mc U_{n}^i = \A^i \hat\Phi_n$. The data sequence $(\f_1, \f_2, \dots )$ represents the evolution of the function $\psi \in \mc H$ due to the nonlinear dynamics $T$ in the coordinates given by the basis $(\phi_1,\dots, \phi_n)$. This representation is amenable to the DMD algorithms we discuss in this paper.

The computed eigenvalues and eigenvectors (eigenmodes) of $\A$ are the key ingredients of the Dynamic Mode Decomposition (DMD), introduced by Schmid \cite{schmid2010}.  Schmid's algorithm is widely used and it has become one of the tools of trade in analysis of fluid flows.
One of its features, stressed both in applications and the development of \emph{Schmid type} DMD methods is the low dimensional approximation of the data using the Singular Value Decomposition (SVD).

\begin{remark}\label{zd:REM:errors}
Note that DMD produces approximate eigenpairs of $\A$ with an error (that depends on the details of a particular implementation), and that the overall error with respect to some eigenvalues of $\mc U$ (part of its point spectrum) depends on the discretization, i.e. on the choice of the finite dimensional subspace $\mc H_n$. In this work, we do not consider the discretization error. 
\end{remark}

\subsection{Preliminaries}
To set the stage, introduce notation and for the reader's convenience, we briefly review some basic facts about Krylov subspaces in eigenvalue computations, and on SVD based low rank approximation. For more details and deeper insights we refer to \cite{Stewart:Book}, \cite{Watkins:Book}, \cite{Liesen-Strakos:Book}.
 
\subsubsection{Krylov decomposition}\label{SSS=Krylov-decomposition}
For $i=1,2,\ldots, m$, define the Krylov matrices 
\begin{equation}\label{zd:eq:XY}
\X_i = \begin{pmatrix} \f_1 & \f_2 & \ldots & \f_{i-1} & \f_i\end{pmatrix},\;\;
\Y_i = \begin{pmatrix} \f_2 & \f_3 & \ldots & \f_{i} & \f_{i+1}\end{pmatrix} \equiv \A \X_i,
\end{equation}
and the corresponding Krylov subspaces $\mathcal{X}_i=\mathrm{range}(\X_i)\subset\mathbb{C}^n$. From the assumption $n\gg m$, $\X_i$ and $\Y_i$ are tall and skinny matrices. The space $\mathbb{C}^n$ is endowed with the complex Euclidean structure; the inner product is $(x,y)=y^* x$, the corresponding norm is $\|x\|_2=\sqrt{(x,x)}$, and the induced matrix (operator) norm is $\|\A\|_2=\max_{\|x\|_2=1}{\|\A x\|_2}$.  The orthogonal projection onto the subspace $\mathcal{X}_i$ is denoted by $\mathbf{P}_{\mathcal{X}_i}$.

Assume that at the index $m$,  $\X_m$ is of full column rank. This implies that 
 \begin{equation}\label{zd:eq:X-flag}
 \mathcal{X}_1\varsubsetneq \mathcal{X}_2\varsubsetneq \cdots\varsubsetneq
 \mathcal{X}_i \varsubsetneq
 \mathcal{X}_{i+1}\varsubsetneq\cdots \varsubsetneq \mathcal{X}_m \varsubsetneq\cdots \varsubsetneq \mathcal{X}_{\ell}=\mathcal{X}_{\ell+1},
 \;\; \A\mathcal{X}_{\ell}\subseteq \mathcal{X}_{\ell},
 \end{equation}
 i.e. $\mathrm{dim}(\mathcal{X}_i)=i$ for $i=1,\ldots, m$, and there must be the smallest saturation index $\ell$ at which $\mathcal{X}_{\ell}=\mathcal{X}_{\ell+1}$. It is well known that then $\mathcal{X}_{\ell}$ is the smallest $\A$-invariant subspace that contains $\f_1$. 

Obviously, with given $\f_1, \ldots, \f_{m+1}$, the action of $\A$ on the range $\mathcal{X}_m$ of $\X_m$ is known, as $\A (\X_m v) = \Y_m v$ for any $v\in\mathbb{C}^m$. Hence, useful spectral information can be obtained using the computable restriction $\mathbf{P}_{\mathcal{X}_m} \restr{\A}{\mathcal{X}_m}$, that is, the Ritz values and vectors extracted using the Rayleigh quotient of $\A$ with respect to $\mathcal{X}_m$.

To that end, let the vector $c=(c_i)_{i=1}^m$ be  computed from the least squares approximation,
	\begin{equation}\label{zd:eq:r=x-Kc}
	{c = \argmin_{v\in\C^m} \| \f_{m+1} - \X_m v\|_2 },
	\end{equation}
and let $r_{m+1} = \f_{m+1} - \X_m c$ be the corresponding residual. Recall that, by virtue of the {projection theorem},  $\X_m c = \mathbf{P}_{\mathcal{X}_m} \f_{m+1}$ and that $r_{m+1}$ is orthogonal to the range of $\X_m$, $\X_m^* r_{m+1}=0$.
Then, since $\f_{i+1}=\A \f_{i}$, $i=1,\ldots, m$, and $\f_{m+1} = \X_m c + r_{m+1}$, we have the Krylov decomposition
\begin{equation}\label{zd:eq:AK=KC+R}
\A \X_m = \X_m C_m + E_{m+1},\;\;  C_m = \begin{pmatrix} 0 & 0 & \ldots & 0 & c_1 \cr
1 & 0 & \ldots & 0 & c_2 \cr 
0 & 1 & \ldots & 0 & c_3 \cr
\vdots & \ddots & \ddots & \vdots & \vdots \cr
0 & 0 & \ldots & 1 & c_{m}\end{pmatrix},\;\;E_{m+1} = r_{m+1} e_m^T ,\;\; e_m=\begin{pmatrix} 0 \cr 0  \cr \vdots \cr 0 \cr 1\end{pmatrix} .
\end{equation}
{Clearly, $C_m=\arg\min_{B\in\C^{m\times m}} \| \A\X_m - \X_m B\|_F$; cf. \S \ref{SS=DMD-Koopman-connection}.}

\noindent The following theorem summarizes well known facts on spectral approximation from $\mathcal{X}_m$:

\begin{theorem}\label{zd:TM:KRR-summary} 
Assume (\ref{zd:eq:XY}), (\ref{zd:eq:X-flag}), (\ref{zd:eq:r=x-Kc}), (\ref{zd:eq:AK=KC+R}), and let $\X_m$ be of full column rank. Then 
\begin{enumerate}
\item The companion matrix 
$C_m=(\X_m^* \X_m)^{-1}(\X_m^* \A \X_m)\equiv \X_m^{\dagger} \A \X_m = (\X_m^* \X_m)^{-1}(\X_m^* \Y_{m})$
is the Rayleigh quotient, i.e. the matrix representation of $\mathbf{P}_{\mathcal{X}_m} \restr{\A}{\mathcal{X}_m}$ 
in the Krylov basis $\X_m$ of $\mathcal{X}_m$. Here $\X_m^\dagger$ denotes the Moore-Penrose generalized inverse of $\X_m$.
\item If $r_{m+1}=0$ (and thus $E_{m+1}=0$ and $m=\ell$ in (\ref{zd:eq:X-flag})) then $\A \X_m = \X_m C_m$ and each eigenpair $C_m w= \lambda w$ of $C_m$ yields an eigenpair of $\A$, $\A(\X_m w) = \lambda (\X_m w)$. 
\item If $r_{m+1}\neq 0$, then $(\lambda, z\equiv \X_m w)$ is an approximate eigenpair,
$\A (\X_m w) = \lambda (\X_m w) + r_{m+1} (e_m^T w)$, {i.e.} $\A z = \lambda z + r_{m+1} (e_m^T w)$.
The Ritz pair $(\lambda, z)$ is acceptable if the residual 
\begin{equation}\label{zd:eq:residual}
\frac{\| \A z - \lambda z\|_2}{\|z\|_2} = \frac{\|r_{m+1}\|_2}{\|z\|_2} |e_m^T w|
\end{equation}
is sufficiently small.
It holds that $z^* r_{m+1}=0$, and 
$\lambda$ is the optimal choice that minimizes the residual (\ref{zd:eq:residual}), i.e.
\begin{equation}\label{eq:optimal-Ritz-value}
\lambda=\frac{z^* \A z}{z^*z}=\argmin_{\zeta\in\mathbb{C}} \| \A z - \zeta z\|_2  
\end{equation}
($\lambda z$ is the orthogonal projection of $\A z$ onto the span of $z$).
\item The residual can be pushed in the backward error, thus making the current Krylov subspace an exactly invariant subspace of a perturbed initial matrix $\A$:
$$
(\A+\Delta \A) \X_m = \X_m C_m,\;\;\mbox{where}\;\; \Delta \A = -r_{m+1} e_m^T(\X_m^* \X_m)^{-1}\X_m^* .
$$
Hence, if $C_m w= \lambda w$, then $(\lambda, \X_m w)$ is an exact eigenpair of $\A+\Delta \A$.
\item If $\A$ is diagonalizable\footnote{In the case of nontrivial Jordan structure of  (non-diagonalizable) $\A$, one can use the theory from \cite{ERXIONG1994691}.} with the eigenvalues $\alpha_1, \ldots, \alpha_n$ and the eigenvector matrix $S$, then for each eigenvalue $\lambda$ of $C_m$,
$
\min_{\alpha_i} |\lambda - \alpha_i| \leq \kappa_2(S) \| \Delta \A\|_2 
$
(Bauer--Fike theorem \cite{Bauer-Fike-1960}) and 
$
\min_{\alpha_i} {|\lambda - \alpha_i|}/{|\lambda|} \leq \kappa_2(S) \| \A^{-1}\Delta \A\|_2
$
(Eisenstat--Ipsen \cite{eis-ipsen-1998}). Here $\kappa_2(S)=\|S\|_2 \|S^{-1}\|_2$, and $\kappa_2(S)=1$ if $\A$ is normal.
\end{enumerate}
\end{theorem}


\begin{remark}\label{zd:Remark-POWIT}
The matrix $\X_m$ can be nearly rank defficient. To illustrate, assume that $\A$ is diagonalizable with eigenpairs $\A \ab_i = \alpha_i \ab_i$, and that its eigenvalues $\alpha_i$ are enumerated so that $0\neq |\alpha_1|\geq |\alpha_2| \geq \cdots\geq |\alpha_n|$. Let $\f_1$ be expressed in the eigenvector basis as $\f_1 =\phi_1 \ab_1 + \cdots + \phi_n \ab_n$. Then
$
\f_{i+1} = \A^i \f_1 = \alpha_1^i \left(\phi_1 \ab_1 + \left({\alpha_2}/{\alpha_1}\right)^i \phi_2 \ab_2 + \left({\alpha_3}/{\alpha_1}\right)^i \phi_3 \ab_3 + \cdots + \left({\alpha_n}/{\alpha_1}\right)^i \phi_n \ab_n \right).
$	
Hence, if e.g. $|\alpha_2| > |\alpha_3|$, then for $j\geq 3$, $\lim_{i\rightarrow\infty}(\alpha_j/\alpha_1)^i = 0$, and thus, with big enough $i$ the $\f_i$'s will stay close to the span of $\ab_1$ and $\ab_2$, provided that $\phi_1\neq 0$, $\phi_2\neq 0$. This means that relatively small changes of $\X_m$ can make it rank deficient; its range may change considerably under tiny  perturbations. In the context of spectral approximations, this is desirable and we hope that the $\f_i$'s will become numerically linearly dependent as soon as possible; on the other hand we must stay vigilant in computing with $\X_m$ and $\Y_m$ as numerical detection of rank deficiency in the presence of noise is a delicate issue. Further, from Theorem \ref{zd:TM:KRR-summary} one can clearly see the advantage of replacing $\X_m$ with an orthonormal matrix, i.e. executing the Rayleigh--Ritz procedure in orthonormal basis.
\end{remark}

\begin{remark}
Clearly, if the subspace $\mathcal{X}_m$ determined as the span of the given dataset does not contain
information on a desired part of the spectrum, then we cannot expect any method to provide detailed insight
into the spectral properties of $\A$. On the other hand, if it does, then we must deploy many different techniques 
to extract relevant spectral information. Any so devised method, in order to be used with confidence, must be accompanied with an error estimate.
\end{remark}

\newpage
\subsubsection{Condition number, SVD and low rank approximations}\label{SSS=SVD-low-rank-approx}
The ill-conditioning of $\X_m$ will pose difficulties. Recall, the spectral condition number of $\X_m$ is defined as
\begin{equation}\label{zd:eq:kappa(X)}
\kappa_2(\X_m) = \|\X_m\|_2 \|\X_m^\dagger\|_2 = 
\frac{\sigma_1}{\sigma_m} ,
\end{equation}
where $\sigma_1\geq\cdots\geq\sigma_m\geq 0$ are the singular values of $\X_m$. High condition number implies closeness to rank deficiency, which is nicely expressed in the following classical theorem.

\begin{theorem}[Eckart-Young \cite{Eckart-Young-1936}, Mirsky \cite{Mirsky-1960}]\label{zd:TM:SVD-EYM}
Let the SVD of $M\in\mathbb{C}^{n\times m}$ be 
$$
M = U\Sigma V^*,\;\;\Sigma = \mathrm{diag}(\sigma_i)_{i=1}^{\min(m,n)},\;\;\sigma_1\geq\cdots\geq\sigma_{\min(m,n)}\geq 0 .
$$	
For $k\in\{ 1, \ldots, \min(m,n)\}$, define $U_k = U(:,1:k)$, $\Sigma_k=\Sigma(1:k,1:k)$, $V_k=V(:,1:k)$, and $M_k=U_k\Sigma_k V_k^*$.	The optimal rank $k$ approximations in $\|\cdot\|_2$ and the Frobenius norm $\|\cdot\|_F$ are
	\begin{eqnarray}
			\min_{\mathrm{rank}(N)\leq k} \| M - N \|_2 &=& \| M - M_k\|_2 = \sigma_{k+1} \label{zd:eq:EYM:2}\\
		\min_{\mathrm{rank}(N)\leq k} \| M - N \|_F &=& \| M - M_k\|_F = \sqrt{\sum_{i=k+1}^{\min(n,m)} \sigma_i^2}  . \label{zd:eq:EYM:F}
	\end{eqnarray}
\end{theorem}
\noindent Hence, if $\sigma_m \ll \sigma_1$, the condition number (\ref{zd:eq:kappa(X)}) is large, $\X_m$ can be made singular with a perturbation $\delta \X_m$ such that $\|\delta \X_m\|_2/\|\X_m\|_2 = \sigma_m/\sigma_1 = 1/\kappa_2(\X_m)\ll 1$. 


This very clearly stresses the importance of the numerical issues, from the purely theoretical questions in perturbation theory to the practical software implementations and computations in the machine finite precision arithmetic.

\subsection{Schmid's DMD method}\label{S=Schmid-method}

The coefficient matrix $\X_m$ in the {least squares} problem (\ref{zd:eq:r=x-Kc}) may be highly ill-conditioned,\footnote{{This is possible even if the underlying $\A$ is unitary.}} and even when the QR factorization $\X_m = Q_m R_m$ is available, it is in general ill-advised to compute $c$ as $c=R_m^{-1}Q_m^* \f_m$, or $C_m$ using the formula from {item} 1.\ in Theorem \ref{zd:TM:KRR-summary} as $C_m = \X_m^\dagger \Y_m= R_m^{-1}Q_m^* \Y_m$ as it has been done e.g. in \cite[Algorithm 1]{Schmid2011}.


To avoid the ill-conditioning, Schmid \cite{schmid2010} used the thin truncated SVD $\X_m = U\Sigma V^* \approx U_k \Sigma_k V_k^*$, where $U_k=U(:,1:k)$ is $n\times k$ orthonormal ($U_k^* U_k=I_k$), $V_k=V(:,1:k)$ is $m\times k$, also orthonormal ($V_k^* V_k=I_k$), and $\Sigma_k=\mathrm{diag}(\sigma_i)_{i=1}^k$ contains the largest $k$ singular values of $\X_m$. 
In brief, $U_k$ is the POD basis for the snapshots $\f_1,\ldots, \f_{m}$. 
Since 
\begin{equation}\label{zd:eq:A*SVDX}
\Y_{m} = \A \X_m \approx \A U_k\Sigma_k V_k^* ,\;\;\mbox{and}\;\; \A U_k = \Y_m V_k \Sigma_k^{-1}, 
\end{equation}
the Rayleigh quotient ${S}_k = U_k^* \A U_k$ with respect to the range of $U_k$ can be computed as
\begin{equation}\label{zd:eq:Schmid-S}
S_k = U_k^* \Y_{m}V_k\Sigma_k^{-1} ,
\end{equation}
which is suitable for data driven setting because it does not use $\A$ explicitly. Clearly, (\ref{zd:eq:A*SVDX}, \ref{zd:eq:Schmid-S}) only require that 
$\Y_m = \A \X_m$; it is not necessary that $\Y_m$ is shifted $\X_m$ as in (\ref{zd:eq:XY}). 
Each eigenpair $(\lambda, w)$ of $S_k$ generates the corresponding Ritz pair $(\lambda, U_k w)$ for $\A$. 
This is the essence of the Schmid's method \cite{schmid2010}, summarized in  Algorithm \ref{zd:ALG:DMD} below.

\begin{algorithm}[hbt]
	\caption{{$[Z_k, \Lambda_k]=\mathrm{DMD}(\X_m,\Y_m)$}}
	\label{zd:ALG:DMD}
	\begin{algorithmic}[1]
		\REQUIRE \  \\		
		\begin{itemize} 
			\item $\X_m=(\x_1,\ldots,\x_m), \Y_m=(\y_1,\ldots,\y_m)\in \mathbb{C}^{n\times m}$ that define a sequence of snapshots pairs $(\x_i,\y_i\equiv \A \x_i)$. (Tacit assumption is that $n$ is large and that $m \ll n$.)
		\end{itemize}
		\STATE $[U,\Sigma, V]=svd(\X_m)$ ; \COMMENT{\emph{The thin SVD: $\X_m = U \Sigma V^*$, $U\in\mathbb{C}^{n\times m}$, $\Sigma=\mathrm{diag}(\sigma_i)_{i=1}^m$, $V\in\mathbb{C}^{m\times m}$}}
		\STATE Determine numerical rank $k$.
		\STATE Set $U_k=U(:,1:k)$, $V_k=V(:,1:k)$, $\Sigma_k=\Sigma(1:k,1:k)$ 	
		\STATE ${S}_k = (({U}_k^* \Y_m) V_k)\Sigma_k^{-1}$; \COMMENT{\emph{Schmid's formula for the Rayleigh quotient $U_k^* \A U_k$}}
		\STATE $[W_k, \Lambda_k] = \mathrm{eig}(S_k)$ \COMMENT{$\Lambda_k=\mathrm{diag}(\lambda_i)_{i=1}^k$; $S_k W_k(:,i)=\lambda_i W_k(:,i)$; $\|W_k(:,i)\|_2=1$}
		\STATE $Z_k = U_k W_k$ \COMMENT{\emph{Ritz vectors}}
		\ENSURE $Z_k$, $\Lambda_k$
	\end{algorithmic}
\end{algorithm}

\section{New approach to computing the DMD}\label{S=New-DMD-RRR}
		
Our goal is to devise a robust software tool for DMD, that will be capable of producing reliable results even in cases where the data and the output vary over several orders of magnitude. To that end, we first review some details from Algorithm \ref{zd:ALG:DMD}, and then we propose some improvements. In particular, we enhance the algorithm with a computable residual error bound, as well as with refinement of the Ritz pairs. These techniques are well known in the projection based large scale eigenvalue computation, and we just adapt them to the data driven framework. Finally we propose certain scalings of the data.
 
\subsection{Preliminaries}
For the sake of completeness and for the reader's convenience, we recall some well-known facts and provide few technical details on the structure of the DMD algorithm. 
\subsubsection{Choosing the dimension of the POD basis}\label{SSS=Choose-k}

The range of the POD basis $U_k$, among all $k$-dimensional spaces, best captures the snapshots in the least squares sense. Namely, if $W$ is any $n\times k$ matrix with orthonormal columns ($W^* W=I_k$) then,  based on Theorem \ref{zd:TM:SVD-EYM},  
\begin{eqnarray*}
	\sum_{i=1}^m \| \f_i - WW^* \f_i \|_2^2 &=& \|\X_m - W(W^*\X_m)\|_F^2\geq  \;\;\mbox{(since $\mathrm{rank}(WW^*\X_m)\leq k$)}\\ &\geq&  \|\X_m - U_k\Sigma_k V_k^*\|_F^2 = \|\X_m - U_kU_k^*\X_m\|_F^2 = 
	\sum_{i=1}^m \| \f_i - U_k U_k^* \f_i \|_2^2 
	= \sum_{i=k+1}^m \sigma_i^2.
\end{eqnarray*}	
This is, of course, the PCA \cite{PCA} of the data. It is interesting to note that this optimal subspace may not contain any of the $\f_i$'s.

The value of $k$, representing the numerical rank of $\X_m$ is determined by inspecting the singular values $\sigma_1\geq\cdots\geq\sigma_m\geq 0$ of $\X_m$ 
and determining $k$ as the largest index such that $\sigma_k > \sigma_1 \epsilon$, i.e.
\begin{equation}\label{zd:eq:k}
k = \max\{ i \; : \; \sigma_i > \epsilon \sigma_1  \},
\end{equation}
where $\epsilon\in (0,1)$ is user supplied tolerance. This is justified by (\ref{zd:eq:EYM:2}) in Theorem \ref{zd:TM:SVD-EYM}. Alternatively, we can use (\ref{zd:eq:EYM:F}) and define
$k = \max\{ i \; : \; \sum_{j=i}^m\sigma_j^2 > \epsilon^2 \sum_{j=1}^m\sigma_j^2  \}$. See \cite{Golub:1976:RDL:892104} for an in depth analysis.

Choosing an appropriate threshold $\epsilon$ is nontrivial in the case of noisy data, and it requires additional case-by-case basis information; see e.g. an analysis in the case of particle image velocimetry data \cite{Epps2010}. For more details see e.g.  \cite[\S 8.2]{dmd-book-kutz-2016}, \cite{Dawson2016}. In this paper we do not consider those issues and focus to the computational aspects of the DMD (see \S \ref{zd:S=General-weighted}), assuming  that the threshold $\epsilon$ (or some other strategy for choosing $k$) is given. 

\subsubsection{Structure of the Rayleigh quotient}
The following proposition explains the relation between $C_m$ and $S_k$, and reveals the details behind the formula for $S_k$.

\begin{proposition}
	Let $S_k$ be computed as in Algorithm \ref{zd:ALG:DMD}, with $\X_m$ and $\Y_m$ as in (\ref{zd:eq:XY}). If $k=m$, then $S_m$ and $C_m$ are similar matrices, where the similarity is realized by the matrix $V\Sigma^{-1}$. 
	If $k<m$, then $S_k$ is the Rayleigh quotient of $C_m$, with the matrix  $V_k\Sigma_k^{-1}$.  
\end{proposition}
\begin{proof}
If $k=m$, then (\ref{zd:eq:AK=KC+R}) yields
$$
\A U_m\Sigma_m V_m^* = U_m\Sigma_m V_m^* C_m + r_{m+1}e_m^T \Longrightarrow S_m\equiv  U_m^*\A U_m = \Sigma_m V_m^* C_m V_m\Sigma_m^{-1} \equiv \Sigma V^* C_m V\Sigma^{-1} .
$$ 
In other words, the Rayleigh quotient is computed using another basis for $\mathcal{X}_m$, necessarily yielding a similar matrix. {A} similar observation in \cite{Jovanovic-Schmid-Nichols-2012} is justified only for the full rank case $k=m$. 

On the other hand, if 
$k<m$ then $\X_m = U_k \Sigma_k V_k^* + \delta \X_m$ where $U_k\Sigma_k V_k^*$ is the best rank $k$ approximation (in the sense of Theorem \ref{zd:TM:SVD-EYM}) of $\X_m$ and $\delta\X_m = \sum_{i=k+1}^m \sigma_i U(:,i)V(:,i)^*$. Hence, $\|\delta \X_m\|_2 = \sigma_{k+1}$ and $U_k^*\delta \X_m=0$, $\delta \X_m V_k=0$. (Note that this implies $\A U_k=\Y_m V_k\Sigma_k^{-1}$ as in (\ref{zd:eq:A*SVDX}). In fact, the formula 
(\ref{zd:eq:A*SVDX}) is in \cite{schmid2010}{, but} derived only for $k=m$. Here we see that its validity for $k<m$ is based on these orthogonality relations\footnote{In finite precision computation this orthogonality is only numerical, i.e. up to rounding errors that depend on a particuar algorithm.} between the truncated part $\delta\X_m$ and the leading left and right singular vectors of $\X_m$. Also note that $U_m^* r_{m+1}=0$ and $(\delta \X_m)^* r_{m+1}=0$.)
In terms of this low rank approximation of $\X_m$,
relation (\ref{zd:eq:AK=KC+R}) reads
$$
\A (U_k \Sigma_k V_k^* + \delta \X_m) = (U_k \Sigma_k V_k^* + \delta \X_m)C_m + r_{m+1}e_m^T,
$$
i.e.
$
\A (U_k \Sigma_k V_k^*) = (U_k \Sigma_k V_k^*) C_m + r_{m+1}e_m^T + \delta \X_m C_m - \A \delta\X_m,
$
and thus (since $\delta \X_m V_k=0$)
\begin{equation}\label{eq:zd:Krylov-dec-Uk-perturbed}
\A U_k = U_k (\Sigma_k V_k^* C_m V_k \Sigma_k^{-1}) + r_{m+1} e_m^T V_k \Sigma_k^{-1} + \delta \X_m C_m V_k\Sigma_k^{-1}.
\end{equation}
In this case, $S_k = U_k^* \A U_k = \Sigma_k V_k^* C_m V_k \Sigma_k^{-1} $ is a Rayleigh quotient of $C_m$.
\end{proof}

\begin{remark}
Note that in relation (\ref{eq:zd:Krylov-dec-Uk-perturbed})
\begin{equation}\label{zd:eq:Schmid-decomp}
\A U_k = U_k S_k + r_{m+1} g_k^* + G_k,\;\;g_k^* =  e_m^T V_k \Sigma_k^{-1},\;\; G_k = \delta \X_m C_m V_k\Sigma_k^{-1} ,
\end{equation}
where
$\|G_k\|_2 = \| \delta \X_m C_m V_k\Sigma_k^{-1}\|_2 \leq \|C_m\|_2 \frac{\sigma_{k+1}}{\sigma_k}$.
This means that neglecting $G_k$ and using the approximate Krylov decomposition\footnote{See \cite{Stewart:KrylovSchur}.}
$\A U_k \approx U_k S_k + r_{m+1} g_k^*$ 
is acceptable only if the singular values are distributed so that
$\sigma_1 \geq \sigma_2 \geq\cdots\geq\sigma_k \gg \sigma_{k+1}\geq\cdots\geq\sigma_m$,
i.e. only in the case of sharp drop after the index $k$. This much stronger truncation criterion is not taken into account in (\ref{zd:eq:k}).
\end{remark}

\subsection{Residual estimates  -- data driven residual computation}\label{zd:SSS:reziduals}

Not all computed Ritz pairs will provide good approximations of eigenpairs of the underlying $\A$, and it is desirable that each pair is accompanied with an error estimate that will determine whether it can be accepted and used in the next steps of a concrete application. The residual is computationally feasible and usually reliable measure of fitness of a Ritz pair. 
 With a simple modification, Algorithm \ref{zd:ALG:DMD} can be enhanced with residual computation, without using $\A$ explicitly. 
 \begin{proposition}\label{zd:PROP:DMD-comp-residual}
For the Ritz pairs $(\lambda_i, Z_k(:,i)\equiv U_k W_k(:,i))$, $i=1,\ldots, k$, computed in Algorithm \ref{zd:ALG:DMD}, the residual norms can be computed as follows:
\begin{equation}\label{zd:eq:DMD-comp-residual}
r_k(i) = \| \A (U_k W_k(:,i)) - \lambda_i (U_k W_k(:,i))\|_2 =
\| (\Y_m V_k\Sigma_k^{-1}) W_k(:,i) - \lambda_i Z_k(:,i) \|_2 .
\end{equation}
Further, if  $\A=S\mathrm{diag}(\alpha_i)_{i=1}^n S^{-1}$, then 
$\min_{\alpha_j} |\lambda_i - \alpha_j| \leq \kappa_2(S) r_k(i)$ (see Theorem \ref{zd:TM:KRR-summary}).
\end{proposition}

Since the formula (\ref{zd:eq:DMD-comp-residual}) requires the matrix $ \A U_k \equiv B_k\equiv \Y_m V_k\Sigma_k^{-1}$, the order of computation must be changed to compute (and store for later use) $B_k$ at the cost of $2nmk + mk$ flops; then, computing $S_k = U_k^* B_k$ takes additional $2nk^2$ flops.
On the other hand, the flop count of the computation of $S_k$ in line 4 {in Algorithm \ref{zd:ALG:DMD}}, as indicated by the parentheses, can be estimated at $2nmk+2mk^2+k^2$ operations. Note that computing  $Z_k$ in line 6  takes $2nk^2$ flops.
Hence, this additional computation of the residuals only mildly increases the overall complexity, but we consider this information an important part of the output and thus the overhead it incurs as justifiable.

In a data driven setting, the right hand side in (\ref{zd:eq:DMD-comp-residual}) is the best one can do. Unfortunately, in finite precision computation, the formula uses computed quantities and it may fail. We discuss this problem and how to fix it in \S \ref{SS=Scaling}, \S \ref{SS=DMD-RRR}, \S \ref{zd:SSS=Synth-Ex-1.2}, \S \ref{SSS=Discussion}. 

\subsection{Data driven refinement of Ritz vectors}\label{zd:SS=refined-vectors}
It is known that the Ritz vectors are not optimal eigenvectors approximations from a given subspace $\mathcal{U}_k=\mathrm{range}(U_k)$. Hence, for a computed Ritz value $\lambda$, instead of the associated Ritz vector, we can choose a vector $z$ that minimizes the residual (\ref{zd:eq:residual}). From the variational characterization of the singular values \cite[Theorem 3.1.2]{horn-johnson-TMA-91}, it follows that
\begin{equation}\label{zd:eq:refined-Ritz}
\min_{z\in\mathcal{U}_k\setminus\{0\}} \frac{\| \A z - \lambda z\|_2}{\|z\|_2} = \min_{w\neq 0} \frac{\| \A U_k w - \lambda U_k w\|_2}{\|U_k {w}\|_2} = \min_{\|w\|_2=1}\| (\A U_k - \lambda U_k)w\|_2 = \sigma_{\min}(\A U_k - \lambda U_k),
\end{equation}
where $\sigma_{\min}(\cdot)$ denotes the smallest singular value of a matrix, and the minimum is attained at the right singular vector $w_{\lambda}$ corresponding to $\sigma_{\lambda}\equiv \sigma_{\min}(\A U_k - \lambda U_k)$. As a result, the refined Ritz vector corresponding to $\lambda$ is $U_k w_{\lambda}$ and the optimal residual is $\sigma_{\lambda}$. For a thorough analysis of the refined Ritz vectors and convergence issues we refer the reader to \cite{Jia19971}, \cite{Jia1999191}, \cite{JIA2001813}, \cite{Jia-Stewart-2001}, \cite{Jia-Sci-China-2004}. 

Here, we need to adapt the refinement procedure to the data driven framework. Using (\ref{zd:eq:A*SVDX}), the minimization of the residual (\ref{zd:eq:refined-Ritz}) can be replaced with computing the smallest singular value with the corresponding right singular vector of $ B_k- \lambda U_k$, where $B_k=\Y_m V_k\Sigma_k^{-1}$. 

Since  (\ref{zd:eq:refined-Ritz}) requires singular value and vector computation of an $n\times k$ matrix for $\lambda=\lambda_i$, $i=1,\ldots, k$, refined Ritz vectors come with an additional computational cost. It can be alleviated using the following preprocessing that replaces the dimension $n$ with much smaller value $2k$:\footnote{Recall that $k\leq m \ll n$.} 

Compute the QR factorization
\begin{equation}\label{zd:eq:refine-QR}
\begin{pmatrix} U_k & B_k \end{pmatrix} = Q R,\;\; R = \bordermatrix{~ & k & k \cr
	k & R_{[11]} & R_{[12]} \cr
	k' & 0 & R_{[22]} \cr} ,\;\;k'=\min(n-k,k), 
\end{equation}
and write the pencil $B_k - \lambda U_k$ as 
\begin{equation}\label{zd:eq:shifted-QR}
B_k - \lambda U_k = Q \left( \begin{pmatrix} R_{[12]} \cr R_{[22]}\end{pmatrix} - \lambda \begin{pmatrix} R_{[11]} \cr 0 \end{pmatrix}\right) \equiv Q R_{\lambda},\;\;\;R_{\lambda} = \begin{pmatrix} R_{[12]} - \lambda R_{[11]} \cr R_{[22]} \end{pmatrix}.
\end{equation}
Obviously, the problem (\ref{zd:eq:refined-Ritz}) is reduced to computing the smallest singular value and the corresponding right singular vector of the $2k\times k$ matrix $R_{\lambda}$.
\begin{proposition}\label{zd:PROP:refined-residual}
For the Ritz value $\lambda=\lambda_i$, let $w_{\lambda_i}$ denote the right singular vector of the smallest singular value $\sigma_{\lambda_i}$ of the matrix $R_{\lambda_i}$ defined using  (\ref{zd:eq:shifted-QR}) and (\ref{zd:eq:refine-QR}). Then the vector $z=z_{\lambda_i}\equiv U_k w_{\lambda_i}$ minimizes the residual (\ref{zd:eq:refined-Ritz}), whose minimal value equals $\sigma_{\lambda_i}=\|R_{\lambda_i}w_{\lambda_i}\|_2$.
\end{proposition}

 Since $k$ is usually small, the cost of preforming this computation with $R_{\lambda_i}$, $i=1,\ldots, k$, is almost negligible in comparison to the overall cost (see \S \ref{zd:SSS:reziduals}). {In an application, we will typically  refine only some of the Ritz pairs, selected by the computed residual and/or applying particular criteria on the Ritz values (e.g. largest real parts, or in absolute value greater than one etc.) and the Ritz vectors (e.g. roughness).}  The matrix $Q$ in (\ref{zd:eq:refine-QR}) is not needed and the computation of $R$ takes 
$ 8 n k^2 - 16 k^3/3$ flops if we take no advantage of the fact that the columns of $U_k$ are already orthonormal.\footnote{Here, for the sake of brevity, we skip technical details on efficient software implementation. They will be available in the accompanying blueprints of the upcoming software.}

The overhead incurred by the QR factorization (\ref{zd:eq:refine-QR}) can be reduced using an interesting and practically useful relation, revealed in the following proposition.
\begin{proposition}\label{zd:prop:S=R22}
In the QR factorization (\ref{zd:eq:refine-QR}), define $\Phi=\mathrm{diag}((R_{[11]})_{ii})_{i=1}^k$. Then
\begin{equation}
\Phi^* R_{[12]} = U_k^* B_k \equiv S_k.
\end{equation}	
\end{proposition}
\begin{proof} If we partition the unitary matrix $Q$ in (\ref{zd:eq:refine-QR}) as $Q=\begin{pmatrix} Q_1, & Q_2\end{pmatrix}$, where $Q_1$ has $k$ columns, then $U_k = Q_1 R_{[11]}$ -- this is a QR factorization of $U_k$. Since $U_k^* U_k=I_k$, and since the QR factorization is essentially unique, $\Phi\equiv R_{[11]}$ is {a} diagonal unitary matrix. The unique QR factorization of $U_k$,  $U_k = U_k I_k = (Q_1 \Phi)(\Phi^* R_{[11]})=(Q_1 \Phi) I_k$, yields $U_k = Q_1 \Phi$. Hence
$
U_k^* B_k = \Phi^* Q_1^* B_k = \Phi^* R_{[12]} .
$
\end{proof}
\begin{remark}
The benefit of Proposition \ref{zd:prop:S=R22} is that it saves $2nk^2$ flops needed to compute $S_k$; instead, at most $k^2$ complex sign changes in $R_{[12]}$ are needed.
Namely, in the numerical software (e.g. Matlab, LAPACK) the QR factorization is implemented so that the diagonal entries of $R$ are real even for complex input matrices. Hence, in practical computation the matrix $\Phi$ is $\mathrm{diag}(\pm 1)_{i=1}^k$.	
\end{remark}


\subsection{Preprocessing raw data: scaling}\label{SS=Scaling}

In Algorithm \ref{zd:ALG:DMD}, the information carried by the $\f_i$'s that are of very small norm (i.e. much smaller than $\max_i \|\f_i\|_2$) is suppressed by truncating small singular values. {Unless the underlying $\A$ is unitary}, we may expect that the column norms of $\X_m$ vary over several orders of magnitude, 
implying large spectral condition number $\kappa_2(\X_m)\equiv \sigma_1 / \sigma_m$. If $\|\f_i\|_2$, $i=1,2,\ldots, m$, e.g. decay very rapidly, then Algorithm \ref{zd:ALG:DMD} will tend to use small $k$. This may be undesirable in case when small norm of $\f_i$ is in the nature of the information it carries, and perhaps just a consequence of a sudden change/switching in the underlying dynamics, and not because it is just a noise. Also, the data $\X_m$, $\Y_m$ can be aggregated from several experiments under different conditions.

As another example, consider computing empirical Gramians of an LTI dynamical systems. Tu and Rowley \cite{Tu20125317} use a combination of spectral projection onto the subspace of slow eigenvalues and an analytical treatment of impulse response tails; DMD is used to identify slow eigenvalues and the snapshots naturally appear as scaled by the weights from the quadrature formula, see \cite[\S 3.2]{Tu20125317}. 

In other words, the subspace $\mathcal{X}_m$ may contain valuable spectral information that will not be captured by the low rank approximation $U_k$ as described in \S \ref{SSS=SVD-low-rank-approx} and \S \ref{SSS=Choose-k}. In that case no method can extract any useful approximation from $U_k$ -- it is simply gone with the truncation. We stress here an important fact: numerical rank determination is not a simple mechanical procedure of truncating the SVD, and, depending on concrete situation, additional information and appropriate procedure may yield better results. 

On the other hand, in the Rayleigh-Ritz procedure, it is the subspace contents that matters and we can additionally process the data
to facilitate better numerical robustness as long we do not change the underlying subspaces.  One simple operation is scaling. Note that $\Y_m = \A \X_m$ is equivalent to $(\Y_m D) = \A (\X_m D)$ for any nonsingular (diagonal) $m\times m$ matrix $D$. Although all these representations of the input-output relations are mathematically equivalent, representing the same physics, they may have different numerical consequences.

Let $\X_m$ be of full column rank, and let $D_x=\mathrm{diag}(\|\X_m(:,i)\|_2)_{i=1}^m$, and let $\X_m = \X_m^{(1)} D_x$. Set $\Y_m = \Y_m^{(1)} D_x$.
Similarly define $D_y=\mathrm{diag}(\|\Y_m(:,i)\|_2)_{i=1}^m$, and set $\Y_m=\Y_m^{(2)} D_y$, $\X_m=\X_m^{(2)}D_y$. Note that $\X_m^{(1)}$ and $\Y_m^{(2)}$ have all columns of unit norm.
An important consequence of this normalization  is that the condition number is nearly optimized over all scalings (see \cite{slu-69})
\begin{equation}
\kappa_2(\X_m^{(1)}) \leq \sqrt{m} \min_{D=\mathrm{diag}}\kappa_2(\X_m D) \leq \sqrt{m} \kappa_2(\X_m) .
\end{equation}
For the sake of brevity, here we consider only the scaling by $D_x$, i.e. we proceed with $\X_m^{(1)}$ and $\Y_m^{(1)}\equiv \A \X_m^{(1)}$ as the new input-output pair. 

A caveat is in order here. Namely, scaling experimental data that might have been contaminated by noise is a tricky business. If for some $\f_i$ with tiny $\|\f_i\|_2$ its SNR is low, then equilibration  has the effect of \emph{noise laundering}, i.e. such a noisy data may \emph{"legally"} participate in numerical rank determination, and it could happen that the POD basis will capture a lot of noise. In such cases, the scaling should be assisted  by statistical reasoning, based on additional input on the statistical properties of the noise; see e.g. \cite{Dawson2016}. In this paper we do not delve into these issues, but we stress their importance that will be addressed in our future work.  
A numerical linear algebra framework for the corresponding computational methods is set in \S  \ref{zd:S=General-weighted}.

\subsection{New method: DDMD\_RRR}\label{SS=DMD-RRR}
As a summary of the preceding considerations, we propose the following Algorithm \ref{zd:ALG:DMD:RRR} -- Refined Rayleigh Ritz Data Driven Modal Decomposition:\footnote{{DDMD\_RRR, to be pronounced as "R-cubed DDMD".}}
 
\begin{algorithm}[H]
	\caption{$[Z_k, \Lambda_k, r_k, \rho_k]=\mathrm{DDMD\_RRR}(\X_m,\Y_m; \epsilon)$ \{\emph{Refined Rayleigh-Ritz DDMD}\}}
	\label{zd:ALG:DMD:RRR}
	\begin{algorithmic}[1]
		\REQUIRE \  \\		
		\begin{itemize} 
			\item $\X_m=(\x_1,\ldots,\x_m), \Y_m=(\y_1,\ldots,\y_m)\in \mathbb{C}^{n\times m}$ that define a sequence of snapshots pairs $(\x_i,\y_i\equiv \A(\x_i))$. (Tacit assumption is that $n$ is large and that $m \ll n$.)
			\item Tolerance level $\epsilon$ for numerical rank determination.		
		\end{itemize}
		\STATE $D_x=\mathrm{diag}(\|\X_m(:,i)\|_2)_{i=1}^m$; $\X_m^{(1)}= \X_m D_x^{\dagger}$; $\Y_m^{(1)} =\Y_m D_x^{\dagger}$
		\STATE $[U,\Sigma, V]=svd(\X_m^{(1)})$ ; \COMMENT{\emph{The thin SVD: $\X_m^{(1)} = U \Sigma V^*$, $U\in\mathbb{C}^{n\times m}$, $\Sigma=\mathrm{diag}(\sigma_i)_{i=1}^m$}}
		\STATE Determine numerical rank $k$, with the threshold $\epsilon$: $k = \max\{ i \; : \; \sigma_i \geq \sigma_1 \epsilon \}$.
		\STATE Set $U_k=U(:,1:k)$, $V_k=V(:,1:k)$, $\Sigma_k=\Sigma(1:k,1:k)$ 	
		\STATE ${B}_k = \Y_m^{(1)} (V_k\Sigma_k^{-1})$; \COMMENT{\emph{Schmid's data driven formula for $\A U_k$}}
		\STATE $[Q, R]=qr(\begin{pmatrix} U_k, & B_k\end{pmatrix})$; \COMMENT{\emph{The thin QR factorization: $\begin{pmatrix} U_k, & B_k\end{pmatrix} = QR$; $Q$ not computed}}
		\STATE $S_k = \mathrm{diag}(\overline{R_{ii}})_{i=1}^k R(1:k,k+1:2k)$ \COMMENT{\emph{$S_k = U_k^* \A U_k$ is the Rayleigh quotient}}
		\STATE $\Lambda_k = \mathrm{eig}(S_k)$ \COMMENT{$\Lambda_k=\mathrm{diag}(\lambda_i)_{i=1}^k$; \emph{Ritz values, i.e. eigenvalues of $S_k$}}
		\FOR{$i=1,\ldots, k$}
		\STATE $[\sigma_{\lambda_i},w_{\lambda_i} ]=svd_{\min}(\left( \begin{smallmatrix} R(1:k,k+1:2k)-\lambda_i R(1:k,1:k)\cr R(k+1:2k,k+1:2k) \end{smallmatrix}\right))$; \COMMENT{\emph{Minimal singular value of $R_{\lambda_i}$ and the corresponding right singular vector}}
		\STATE $W_k(:,i)=w_{\lambda_i}$
		\STATE $r_k(i) = \sigma_{\lambda_i}$ \COMMENT{\emph{Optimal residual, $\sigma_{\lambda_i}=\|R_{\lambda_i} w_{\lambda_i}\|_2$}}
		\STATE $\rho_k(i)=w_{\lambda_i}^* S_k w_{\lambda_i}$ \COMMENT{\emph{Rayleigh quotient, $\rho_k(i)= (U_k w_{\lambda_i})^* \A (U_k w_{\lambda_i})$}}
		\ENDFOR
		\STATE $Z_k = U_k W_k$ \COMMENT{\emph{Refined Ritz vectors}}
		\ENSURE $Z_k$, $\Lambda_k$, $r_k$, $\rho_k$
	\end{algorithmic}
\end{algorithm}

\noindent Few comments are in order.

\begin{remark}
In Line 8, only the eigenvalues are computed, which saves $O(k^3)$ flops as compared to Line 5 in Algorithm \ref{zd:ALG:DMD}. The Ritz vectors are then computed in the refinement procedure. Only this option is shown for the sake of brevity; our software implementation allows first computing the Ritz vectors with the corresponding residuals, and then refining only the selected ones, recall the discussion in \S \ref{zd:SS=refined-vectors}.
\end{remark}	

\begin{remark}
The simplest way to implement Line 10 is to compute the full $2k\times k$ SVD, and then to take the smallest singular value $\sigma_{\lambda_i}$ and the corresponding right singular vector $w_{\lambda_i}$. It is an interesting research problem to develop a more efficient method that will target only $\sigma_{\lambda_i}$, $w_{\lambda_i}$ and thus reduce the complexity, from $O(k^3)$ to $O(k^2)$ per refined vector.	
\end{remark}

\begin{remark}
		Proposition \ref{zd:PROP:refined-residual} identifies the smallest singular value $\sigma_{\lambda_i}$ as the value of the minimal residual. We should be aware that the smallest singular value may be computed with relatively large error, so in practice it may not represent the residual norm exactly. Although we are not interested in the number of accurate digits of {the} residual, but in its order of magnitude, we can also compute it as $\sigma_{\lambda_i}=\|R_{\lambda_i} w_{\lambda_i}\|_2$, or using the formula (\ref{zd:eq:DMD-comp-residual}). 
\end{remark}		

\begin{remark}\label{zd:REM:rho(i)}
	Line 13 is motivated by the fact that for any candidate $x$ for an eigenvector of $\A$, the Rayleigh quotient $\rho=x^* \A x/(x^* x)$ minimizes the residual $\|\A x - \rho x\|_2$; see item 3. in Theorem \ref{zd:TM:KRR-summary} . Hence, for $i=1,\ldots, k$, the pair $(U_k w_{\lambda_i},\rho_k(i))$ can be used instead of $(U_k w_{\lambda_i},\lambda_i)$. 	
\end{remark}

\section{Numerical examples}\label{zd:S=Examples}
The main goal of the modifications of the DMD algorithm, proposed in \S \ref{S=New-DMD-RRR}, is to provide a reliable black-box, data driven software device that can estimate part of 
the spectral information of the underlying linear operator, and that also can  provide an error bound.   In this section we illustrate the performance of the new, modified DMD algorithm, 
first using a synthetic example  (to illustrate the fine details of numerical software development, and possible causes of failure even in seemingly simple cases), and then in a concrete example from computational fluid dynamics (to illustrate the benefits of enhanced functionality of the new DMD algorithm in deeper analysis of the original physical problem and its numerical model).

A series of numerical tests show that Algorithm \ref{zd:ALG:DMD:RRR} is never much worse, but it can be substantially better than Algorithm \ref{zd:ALG:DMD}. We also show how the added functionality of the new algorithm works in practice. 

\subsection{Case study: a synthetic example}\label{zd:SSS=Synth-Ex}
The first  test drive of the code uses randomly generated test matrix that is difficult enough to illustrate the improvements, and to expose weaknesses from which we can learn how {to} devise better methods. Also, we use this example to stress the importance of careful software implementation. We use Matlab 8.5.0.197613 (R2015a).

The test matrix is generated as $A = \mathbf{e}^{-B^{-1}}$ where $B$ is pseudo-random with entries uniformly distributed in $[0,1]$, and then $\A = A /\|A\|_2$. (This normalization is only for convenience, because all involved subspaces remain unchanged, and the Ritz values and the eigenvalues change with the same scaling factor.) Although this example is purely synthetic,  it may represent a situation with the spectrum entirely in the unit disc, such as e.g. in the case of an off-attractor analysis of a dynamical system, after removing the peripheral eigenvalues \cite{Mohr-Mezic-GLA-2014}. 

Then, the initial state $\f_1$ is taken with entries from uniform distribution in $[0,1]$, and the sequence $\f_{i+1}=\A \f_i$ is generated for $i=1,\ldots, m$. The dimensions are taken as $n=1000$, $m=99$. The reference (true) eigenvalues of $\A$ are computed using the Matlab function \texttt{eig()}. The subspace spanned by the snapshots $\f_i$ contains valuable information on the spectrum of $\A$ and we test the ability of a method to extract that information and to correctly report its fidelity.
The residuals are computed exactly (using $\A$, for testing purposes) and also returned by the algorithms (which have no access to $\A$), using only the snapshots and Proposition \ref{zd:PROP:DMD-comp-residual} and Proposition \ref{zd:PROP:refined-residual}.
The truncation threshold in (\ref{zd:eq:k}) is taken as $\epsilon = n \roff$, where $\roff$ is the round-off unit; in Matlab $\roff=\texttt{eps}\approx 2.2\cdot 10^{-16}$.

\subsubsection{Sequential shifted data, numerical rank $k$ defined as in (\ref{zd:eq:k})}\label{zd:SSS=Synth-Ex-1.1} 
We take $\X_m = (\f_1, \f_2, \dots , \f_{99})$, $\Y_m=(\f_2, \f_3,\ldots, \f_{100})$, and give them as inputs to  both
 Algorithm \ref{zd:ALG:DMD} and Algorithm \ref{zd:ALG:DMD:RRR}. The difference in the quality of the computed Ritz pairs, illustrated {in} Figure \ref{zd:fig:Ex1:eigs} and Figure \ref{zd:fig:Ex1:resids}, is striking. 
We first note that DMD has returned only $8$ Ritz pairs, while DDMD\_RRR returns $27$. The reason is in a sharp decay of the norms $\|\f_i\|_2$, forcing the truncated SVD into declaring low numerical rank. As a result, small number of eigenvalues is approximated from a subspace of small dimension, and large portion of the supplied information is truncated, i.e. left unexploited.

\begin{figure}[H]
\begin{center}
\includegraphics[width=\linewidth, height=3in]{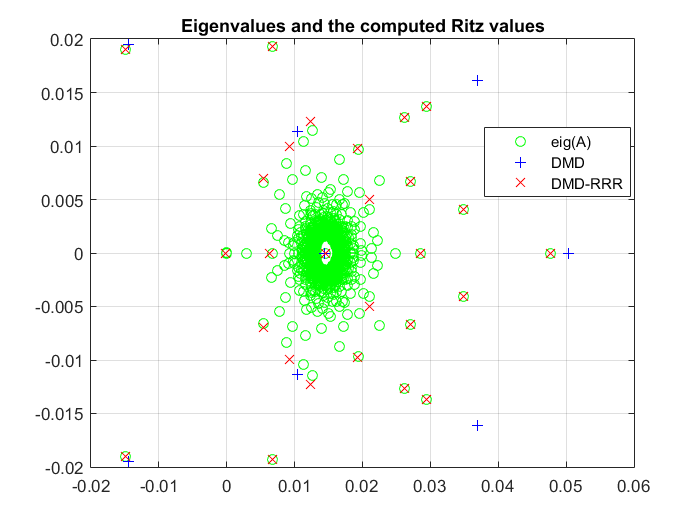}
\end{center}	
\caption{\label{zd:fig:Ex1:eigs} (Example in \S \ref{zd:SSS=Synth-Ex-1.1}, with sequential data) Comparison of the approximations of the eigenvalues of $\A$ (circles $\circ$) by the Ritz values computed by Algorithm \ref{zd:ALG:DMD} (pluses $+$) and Algorithm \ref{zd:ALG:DMD:RRR} (crosses, $\times$), with the same threshold in the truncation criterion for determining the numerical rank as defined in (\ref{zd:eq:k}).}
\end{figure} 

\begin{figure}[H]
	\begin{center}
		\includegraphics[width=\linewidth, height=1.8in]{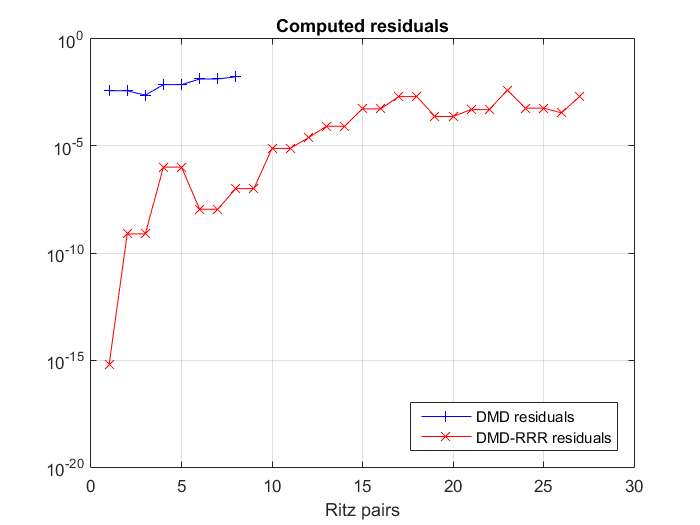}
	\end{center}	
	\caption{\label{zd:fig:Ex1:resids} (Example in \S \ref{zd:SSS=Synth-Ex-1.1}, with sequential data) Comparison of the residuals of the Ritz pairs computed by Algorithm \ref{zd:ALG:DMD} (pluses $+$) and Algorithm \ref{zd:ALG:DMD:RRR} (crosses, $\times$), with the same threshold in the truncation criterion for determining the numerical rank as defined in (\ref{zd:eq:k}).}
\end{figure} 

\subsubsection{Sequential shifted data, numerical rank hard-coded as $k=27$}\label{zd:SSS=Synth-Ex-1.2}
In this experiment, we set $k=27$, so both algorithms will take best $27$-dimensional subspaces from their SVD decompositions; hence Algorithm \ref{zd:ALG:DMD:RRR} will return the same output as in \S \ref{zd:SSS=Synth-Ex-1.1}, and Algorithm \ref{zd:ALG:DMD} is allowed to use the same dimension\footnote{Note that this means that both algorithms use subspaces of same dimensions, but not necessarily the same subspace.} $27$, instead of $8$ as it would have been taken using (\ref{zd:eq:k}).
This is an artificial scenario, as the algorithm must determine the most appropriate value of $k$ autonomously for each given input -- the ramifications of choosing $k$ inappropriately are illustrated in \S \ref{zd:SSS=Synth-Ex-1.1}. However, it is instructive to see and analyze the results returned by Algorithm \ref{zd:ALG:DMD}. 

The computed Ritz values, shown {in} Figure \ref{zd:fig:Ex1:eigs:k27}, reveal an interesting fact -- we see only $12$ of them (pluses, $+$) because many are computed as tiny numbers clustered around zero, and the overall residuals, displayed on Figure \ref{zd:fig:Ex1:resids:k27}, show no substantial improvement.

\begin{figure}[H]
	\begin{center}
		\includegraphics[width=\linewidth, height=3in]{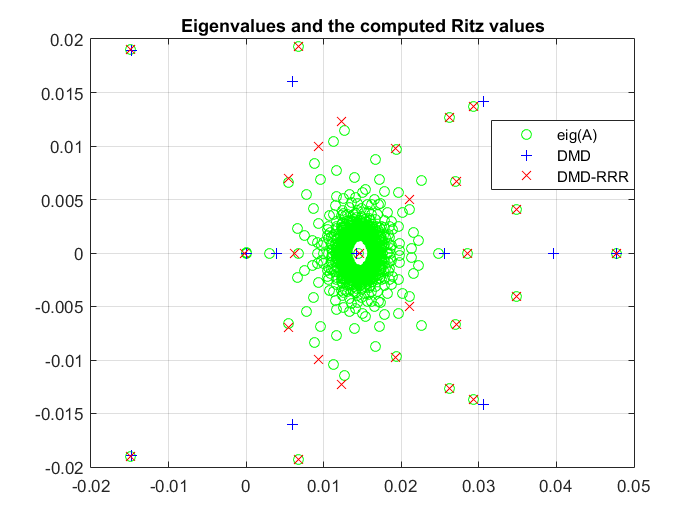}
	\end{center}	
	\caption{\label{zd:fig:Ex1:eigs:k27} (Example in \S \ref{zd:SSS=Synth-Ex-1.2}, with sequential data) Comparison of the approximations of the eigenvalues of $\A$ (circles $\circ$) by the Ritz values computed by Algorithm \ref{zd:ALG:DMD} (pluses $+$) and Algorithm \ref{zd:ALG:DMD:RRR} (crosses, $\times$), with the same  numerical rank $k=27$. The plus in the vicinity of the origin is actually a cluster of $16$ pluses ($+$) clustered around zero with the Ritz values in modulus equal     
\emph{5.4e-05,
	7.1e-07,
	2.0e-08,
	1.4e-10,
	1.1e-12,
	3.8e-14,
	2.8e-16,
	6.5e-19,
	7.5e-20,
	1.9e-22,
	1.0e-24,
	1.0e-24,
	6.9e-28,
	6.6e-31,
	3.4e-31,
	3.6e-33}.}
\end{figure} 

Furthermore, the formula (\ref{zd:eq:DMD-comp-residual}) for the residuals collapsed when used in Algorithm \ref{zd:ALG:DMD}, returning the values that were underestimated by several orders of magnitude, as shown by the ratios 
\begin{equation}\label{zd:eq:rez-ratios}
\eta_i =
\frac{\| (\Y_m V_k\Sigma_k^{-1}) W_k(:,i) - \lambda_i (U_k W_k(:,i)) \|_2}{\| \A (U_k W_k(:,i)) - \lambda_i (U_k W_k(:,i))\|_2} \equiv 1 ,\;\;i=1,\ldots, k.
\end{equation}
The computed values of $\eta_i$ for Algorithm \ref{zd:ALG:DMD} are below $10^{-10}$, i.e. the residuals returned by the algorithm are e.g. $10^{10}$ times smaller than the actual values (see Figure \ref{zd:fig:Ex1:resids:k27}), yielding wrong conclusion that the approximations are good.

\begin{figure}[H]
	\begin{center}
		\includegraphics[width=2.9in, height=2.2in]{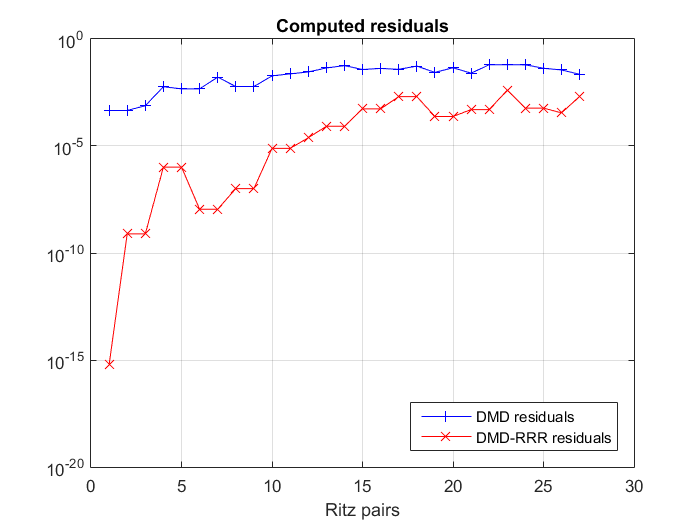}
		\includegraphics[width=2.9in, height=2.2in]{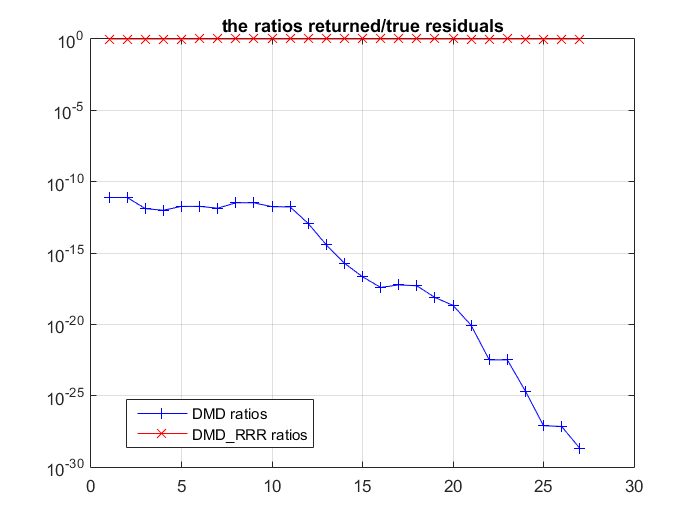}
	\end{center}	
	\caption{\label{zd:fig:Ex1:resids:k27} (Example in \S \ref{zd:SSS=Synth-Ex-1.2}, with sequential data) Comparison of the residuals of the Ritz pairs computed by Algorithm \ref{zd:ALG:DMD} (pluses $+$) and Algorithm \ref{zd:ALG:DMD:RRR} (crosses, $\times$), with the same  numerical rank $k=27$. The second plot shows the ratios (\ref{zd:eq:rez-ratios}) of the returned and the true residuals computed using $\A$ explicitly, as in the denominator in (\ref{zd:eq:rez-ratios}).}
\end{figure} 

\subsubsection{Discussion}\label{SSS=Discussion}
It is important to understand why the DMD did not take advantage of the supplied information on the numerical rank. 
 
\paragraph{What went wrong?} The key insight is in Figure \ref{zd:fig:Ex1:resids:svdX}. Only 10 singular values used in Algorithm \ref{zd:ALG:DMD} (computed in Line 2., using the Matlab function \texttt{svd()}) are accurate to a satisfactory level; the remaining $17$ are mostly severely overestimated, as the true singular values of $\X_m$ decay very rapidly down to the level of $10^{-150}$. In Figure \ref{zd:fig:Ex1:resids:svdX}, we display those values, together with the singular values of $\X_m$, computed by calling {the same Matlab function \texttt{svd()}}. Only the first $10$ singular values agree to some level of accuracy. 
Due to the large upward error, the formula for $\A U_k$ as $\Y_m V_k\Sigma_k^{-1}$ fails and the trailing columns of $\Y_m V_k\Sigma_k^{-1}$ are severely underestimated because the computed $\Sigma_k^{-1}$ is not big enough, see Figure 
\ref{zd:fig:Ex1:resids:UAU}. (We have checked that the minimal singular value of $\A$ is larger than $10^{-6}$, and the minimax theorem implies that the  norms of the columns of $\A U_k$ must be also larger than $10^{-6}$.) As a result, the trailing $17$ columns of te matrix $S_k$ are entirely wrong and  mostly very tiny in norm.

Further, we should keep in mind that determining numerical rank (and truncating the SVD of the data matrix) is  a delicate procedure that must consider also the scaling of the data, statistical information on the error in the  data etc. Hard coded universal threshold, without any data preprocessing,  is not always the best one can do.

\paragraph{Why it went wrong?}
The difference between the computed singular values is due to the fact that Matlab uses different algorithms in the \texttt{svd()} function, depending on whether the singular vectors are requested on output.  The faster but less accurate method is used in the call $[U,S,V]=\mathrm{\texttt{svd}}(\X_m,'econ')$.
To or best knowledge, Matlab documentation does not provide sufficient information, except that the SVD is based on the LAPACK subroutines. It is very likely that the full SVD, including the singular vectors, is computed using the divide and conquer algorithm, \texttt{xGESDD()} in LAPACK and, for computing only the singular values, $S=\texttt{svd}(X)$ calls the QR SVD, \texttt{xGESVD()} in LAPACK. This policy of not specifying the method, and thus hiding the key information on the numerical reliability of the output may cause numerical errors to inconspicuously degrade the overall computation. Without providing information that may be mission critical, software developers often choose to prefer faster method by sacrificing numerical accuracy, without informing the user on possible consequences. Note that the same fast \texttt{xGESDD()} subroutine is (to the best of our knowledge) under the hood of the Python function \texttt{numpy.linalg.svd}.
Numerical robustness of both \texttt{xGESVD()}  \texttt{xGESDD()} depends on $\kappa_2(\X_m)$, and if one does not take advantage of the fact that scaling is allowed, the problems illustrated in this section are likely to happen.
For numerically more robust computation of the SVD independent of scaling, see \cite{drm-ves-VW-1}, \cite{drm-ves-VW-2}, \cite{drm-xgesvdq}.

\begin{figure}[H]
	\begin{center}
		\includegraphics[width=\linewidth, height=2.5in]{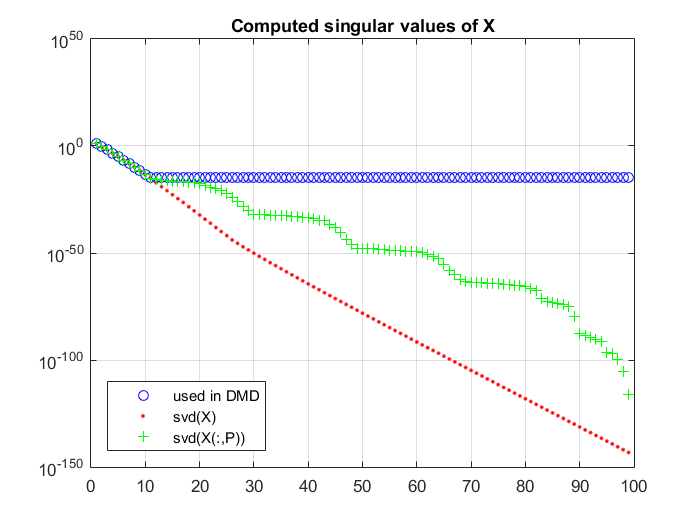}
	\end{center}	
	\caption{\label{zd:fig:Ex1:resids:svdX} (Example in \S \ref{zd:SSS=Synth-Ex-1.2}, with sequential data) Three sets of the computed approximate singular values of the data matrix $\X_m$. The blue circles (\textcolor{blue}{$\circ$}) are the values used in Algorithm \ref{zd:ALG:DMD} and are computed as $[U,\Sigma,V]=\mathrm{\texttt{svd}}(\X_m,'econ')$. The red dots (\textcolor{red}{$\cdot$}) are computed as $\Sigma=\mathrm{\texttt{svd}}(\X_m)$, and the pluses ($+$) are the results of $\Sigma=\mathrm{\texttt{svd}}(\X_m(:,P))$, where $P$ is randomly generated permutation.}
\end{figure} 

\begin{figure}[H]
	\begin{center}
		\includegraphics[width=\linewidth, height=2.5in]{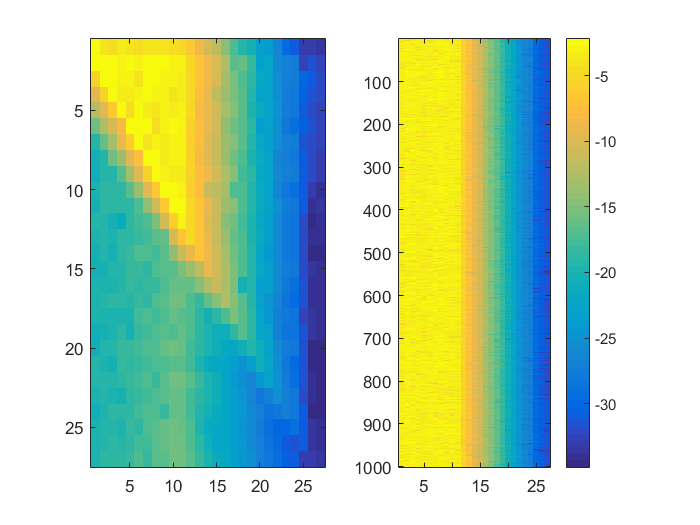}
	\end{center}	
	\caption{\label{zd:fig:Ex1:resids:UAU} (Example in \S \ref{zd:SSS=Synth-Ex-1.2}, with sequential data) Left: The matrix $\log_{10}|S_k|$ as computed in Algorithm \ref{zd:ALG:DMD}. Right: The matrix $\log_{10}|\Y_m V_k\Sigma_k^{-1}|$, using $V_k$ and $\Sigma_k$ as computed in Algorithm \ref{zd:ALG:DMD}, using Matlab  $[U,\Sigma,V]=$svd($\X_m$,'econ'), $V_k=V(:,1:k)$, $\Sigma_k=\Sigma(1:k,1:k)$. Recall that $S_k = U_k^* (\Y_m V_k\Sigma_k^{-1})$.}
\end{figure}
\subsubsection{Odd-even data splitting} We conclude this experiment with another run, but this time we halve the dimensions of the data subspaces by taking 
$$
\X_{m/2}=(\f_1,\f_3,\ldots ,\f_{97},\f_{99}),\;\;\; \Y_{m/2}=(\f_2, \f_4,\ldots, \f_{98},\f_{100}), 
$$
i.e. split the snapshots by taking every other into $\X_{m/2}$. The results are given in Figure \ref{zd:fig:Ex1:eigs-odd-even}. Again, Algorithm \ref{zd:ALG:DMD:RRR} has identified more Ritz pairs with smaller residuals than Algorithm \ref{zd:ALG:DMD}, even with the initial space of half of the dimension, as compared with sequential shifted data. It is worth mentioning that the complexity of the first step in all DMD methods is quadratic in the column dimension, thus the reduction here is with the factor of four.

\begin{figure}[H]
	\begin{center}
		\includegraphics[width=2.9in, height=2.2in]{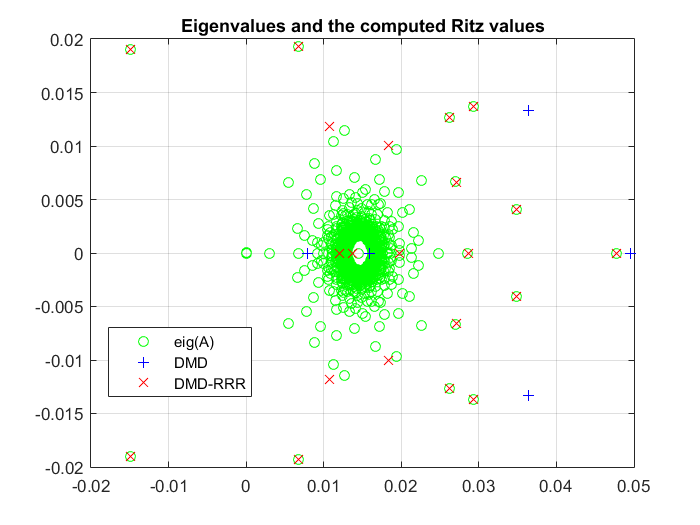}
		\includegraphics[width=2.9in, height=2.2in]{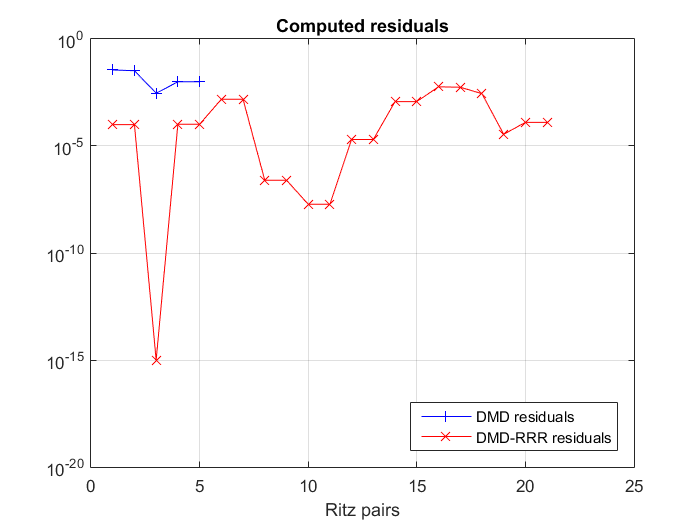}
	\end{center}	
	\caption{\label{zd:fig:Ex1:eigs-odd-even} (Example in \S \ref{zd:SSS=Synth-Ex}, odd-even data splitting) Left: The eigenvalues of $\A$ (circles \textcolor{green}{$\circ$}), the Ritz values computed by Algorithm \ref{zd:ALG:DMD} (pluses \textcolor{blue}{$+$}) and Algorithm \ref{zd:ALG:DMD:RRR} (crosses, \textcolor{red}{$\times$}), with the same threshold in the truncation criterion for determining the numerical rank as defined in (\ref{zd:eq:k}). Right: Comparison of the residuals.}
\end{figure} 


\subsection{Flow around a cylinder}\label{zd:SS=Example-cylnder}
In this example, we illustrate the benefits of the new proposed techniques in a concrete application -- numerical  Koopman spectral analysis of the wake behind a circular cylinder. 
The well studied and understood  model of laminar flow around a cylinder is based on the two-dimensional incompressible Navier-Stokes equations
\begin{eqnarray}
\frac{\partial \mathbf{v}}{\partial t} &=& - (\mathbf{v} \cdot \nabla)\mathbf{v} + \nu \Delta \mathbf{v} - \frac{1}{\rho}\nabla p \\
\label{eq:incompressible_navier_stokes} 
\nabla\cdot\mathbf{v} &=& 0,
\end{eqnarray}
where $\mathbf{v} = (v_{x}, v_{y})$ is velocity field, $p$ is pressure, $\rho$ is fluid density and $\nu$ is kinematic viscosity.
The flow is characterized by the Reynolds number
$\mathfrak{Re}={v^* D}/{\nu}$
where, for flow around circular cylinder, the characteristic quantities are the inlet velocity $v^*$ and the cylinder diameter $D$. For a detailed analytical treatment of the problem see \cite{bagheri-2013}, \cite{glaz2016quasi}; for a more in depth description of the Koopman analysis of fluid flow we refer to \cite{Mezic:2013}, \cite{Rowley:2009ez}. 

We used the velocity and pressure data from a numerical simulation described in detail  in \cite{AIMdyn-cylinder-wake-2017}. We thank Dr. Stefan Ivi\'{c} and Dr. Nelida \v{C}rnjari\'{c}--\v{Z}ic for  providing us with the data. For the sake of completeness and for the reader's convenience, we briefly describe the numerical procedure.

The flow, with $x$ denoting the streamwise direction, is numerically simulated in the rectangular domain $[-12D, 20D] \times [-12D, 12D]$, with appropriate boundary conditions. The inlet velocity is set to
$\mathbf{v}=(v^*,0)$, and the outlet $\partial\mathbf{v}/\partial x=0$; on the side bounds the slip wall condition $v_{y}=0$ is imposed, while the no-slip wall condition $\mathbf{v}=(0,0)$ is set on the cylinder edge. For the pressure, homogeneous Neumann boundary condition is used for all domain boundaries except for the outlet where $p=0$.

With this setup, the laminar incompressible two-dimensional flow model is for $\mathfrak{Re}=150$ simulated numerically using OpenFOAM's \texttt{icoFOAM} solver \cite{jasak2007openfoam} on a structured grid containing 86000 cells, and with time step $\Delta t=0.01~\textmd{s}$.

The observables for the DMD analysis, two components of the velocity and the pressure, are collected as follows:
After the flow was fully developed, the snapshots are extracted from a  time interval of $100$ seconds with the time step of $\Delta t=0.1~s$. All three state variables obtained on the structured CFD mesh are interpolated, and then evaluated on a $201 \times 101$ uniform rectangular grid in the $[-10,30] \times [-10,10]$ sub-domain. The $201 \times 101$ array data structure is then, for each variable, reshaped into an $20301\times 1$ array. 

The thus obtained data matrix of the snapshots has the block row structure with each snapshot $\f_i$ containing as observables the components of the velocities $v_{ix}$, $v_{iy}$, and the pressure $p_i$,
$$
\f_i = \begin{pmatrix} v_{ix} \cr v_{iy} \cr p_i\end{pmatrix} \equiv  \begin{pmatrix} \widehat{\f}_i \cr p_i\end{pmatrix} ,\;\;
v_{ix},  v_{iy},  p_i \in \mathbb{R}^{20301},\;\;i=1,\ldots , 1001.
$$
We also perform the test with only the components of the velocities as observables, i.e. with $\widehat{\f}_i=( v_{ix},  v_{iy})^T$. For an interpretation of the difference between using $\f_i$ and $\widehat{\f}_i$ , recall \S \ref{SS=DMD-Koopman-connection}.

\noindent Our aim with this example is to illustrate:
\begin{itemize}
\item the advantage of using the data driven residuals to select good Ritz pairs;
\item the advantage of combining the computable residual bounds with the spectral perturbation theory to obtain computable a posteriori error bounds in the approximate eigenvalues 
\item the potential of the Ritz vector refinement procedure (reducing the residual with better approximate eigenvectors, and improving the approximate eigenvalues)
\item the effect of choosing and scaling of the observables (in the context of spectral analysis of the underlying Koopman operator)
\end{itemize}

Let us now discuss the results.

Figure \ref{zd:fig:Ex1:resids-cylinder} shows in the left panel the residuals computed as explained in \S \ref{zd:SSS:reziduals}; for both sets of data the computed Ritz pairs are ordered with increasing residuals. This allows automatic selection of those that better approximate the eigenpairs of the underlying operator. For instance, with a threshold for the residual set to $5\cdot 10^{-4}$,  the selected Ritz values are shown on the right panel. We have checked that the selected Ritz values are close to the unit circle, up to an error of order $10^{-5}$, with quite a few that are  up to $O(10^{-7})$ on the circle. The numerical data from the cylinder flow is selected so that the dynamics has stabilized to evolution on the limit cycle attractor. On the attractor, the Koopman operator is unitary with respect to an underlying invariant measure \cite{Mezic:2005}. Thus, the eigenvalues that we obtain are expected to reside close to the unit circle.

As expected, the $\f_i$'s carry more information than the $\widehat{\f}_i$'s, and more Ritz values have been selected.  However, here we should keep in mind that in a numerical scheme the pressure values might be computed less accurately than the velocities, so that the advantage of extra information is not necessarily fully exploited.
The issues of different level of noise in the observables, different physical nature and scalings (units) of the variables, and proper choice of the norm to measure the residual are important; see \S
\ref{zd:S=General-weighted} for an algebraic framework and numerical algorithms.

\begin{figure}[H]
	\begin{center}
		\includegraphics[width=2.9in, height=2.8in]{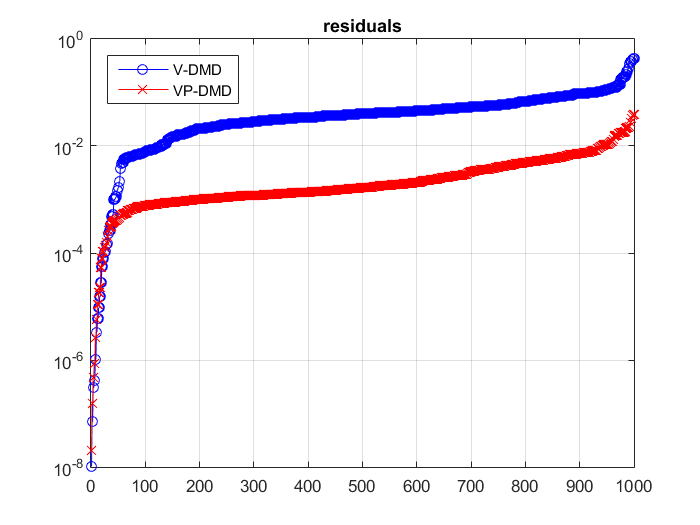}
			\includegraphics[width=2.9in, height=2.8in]{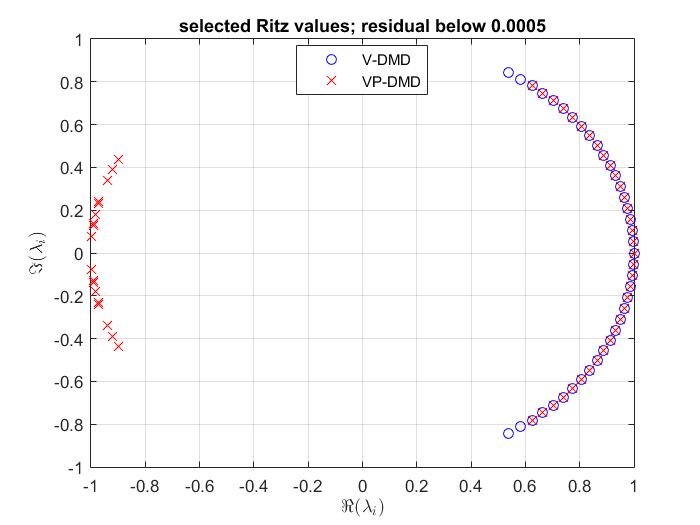}
	\end{center}	
	\caption{\label{zd:fig:Ex1:resids-cylinder} (Example in \S \ref{zd:SS=Example-cylnder}) Left panel: Comparison of the residuals of the Ritz pairs computed by Algorithm \ref{zd:ALG:DMD:RRR} with velocities as observables (V-DMD, circles \textcolor{blue}{$\circ$}) and with both velocities and pressures (VP-DMD, crosses, \textcolor{red}{$\times$}). Right panel:  Selected Ritz values computed by Algorithm \ref{zd:ALG:DMD:RRR} with velocities as observables (circles \textcolor{blue}{$\circ$}) and with both velocities and pressures (crosses, \textcolor{red}{$\times$}).}
\end{figure} 

For the selected pairs, the refinement procedure of \S \ref{zd:SS=refined-vectors} additionally reduces the residuals, as shown on Figure \ref{zd:fig:Ex1:refined-resids-cylinder}.
\begin{figure}[H]
	\begin{center}
		\includegraphics[width=2.9in, height=2in]{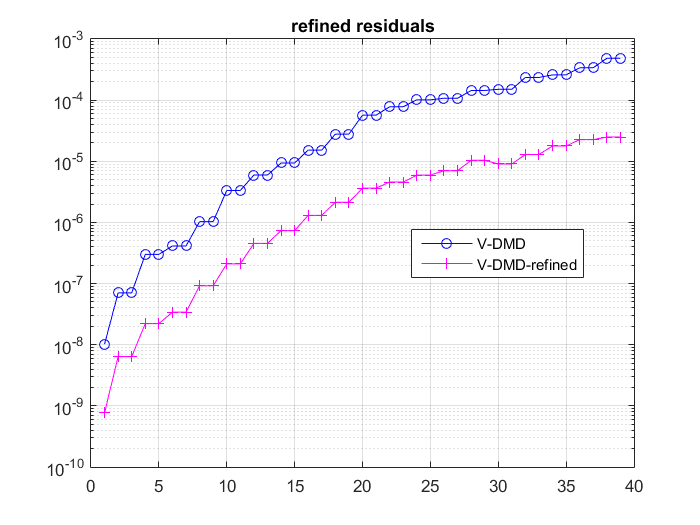}
		\includegraphics[width=2.9in, height=2in]{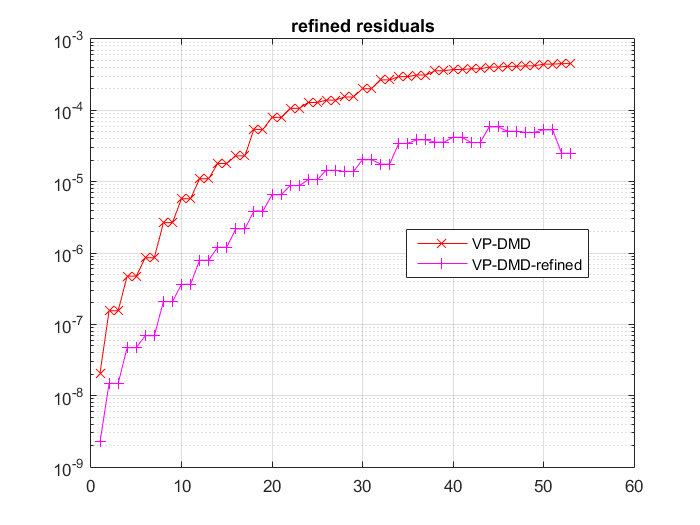}
	\end{center}	
	\caption{\label{zd:fig:Ex1:refined-resids-cylinder} (Example in \S \ref{zd:SS=Example-cylnder}) Comparison of the refined residuals of the Ritz pairs computed by Algorithm \ref{zd:ALG:DMD:RRR} with velocities as observables (top 39 pairs in V-DMD, circles \textcolor{blue}{$\circ$}) and with both velocities and pressures (top 53 pairs in VP-DMD, crosses, \textcolor{red}{$\times$}). The noticeable staircase structure on the graphs corresponds to complex conjugate Ritz pairs.}
\end{figure} 
In fact, since this reduction of the residuals is achieved by improving the Ritz vectors, we can achieve even smaller residuals by refining the Ritz values as in {item} 3. in Theorem \ref{zd:TM:KRR-summary} and Remark \ref{zd:REM:rho(i)}. Furthermore, the classical Bauer--Fike type error bound from {item} 5. in Theorem \ref{zd:TM:KRR-summary} (also used in Proposition \ref{zd:PROP:DMD-comp-residual}) gives useful relative error bounds for the Ritz values close to the unit circle, in particular because the underlying $\A$ is nearly unitary (in an appropriate discretized inner product) for sufficiently large number of data points (see the discussion in \cite{KordaandMezic:2017} on the convergence of finite-dimensional approximations to the Koopman operator, and notice that the underlying Koopman operator for this data is nearly unitary, as discussed above), so the condition number of its eigenvalues is close to one. For instance, we know for sure that our computed values $\lambda_i$ approximate some eigenvalues of $\A$ up to $-\lceil \log_{10}(r_k(i))\rceil$ digits of accuracy\red{;} from Figure \ref{zd:fig:Ex1:refined-resids-cylinder} we see that this means between four and nine correct digits.\footnote{Here we tacitly assume that the residuals are computed sufficiently accurately, and that the error analysis is done using appropriate norms. See Remark \ref{zd:REM:A:M-unitary} and Theorem \ref{zd:TM:BF-weighted}.}

 Without separate analysis we cannot make an analogous statement with respect to the eigenvalues of the underlying Koopman operator; recall Remark \ref{zd:REM:errors}. This important topic of including the operator discretization error in the overall error bound is the subject of our future research.

The approximate Koopman eigenvalues are computed in the vicinity of unit circle as 
\begin{equation}\label{koop-eig}
\koop_j =  \frac{\log\lambda_j}{2\pi\Delta t} \equiv \frac{\log|\lambda_j| + \ii \arg\lambda_j}{2\pi\Delta t}\approx 
\frac{1}{2\pi\Delta t} \ii \arg\lambda_j . 
\end{equation}
Using the selected Ritz values (and their refinements) the computed $\koop_i$'s are as in Figure \ref{zd:fig:Ex1:Koop-eigs0-cylinder}.

\begin{figure}[H]
	\begin{center}
					\includegraphics[width=2.90in, height=3in]{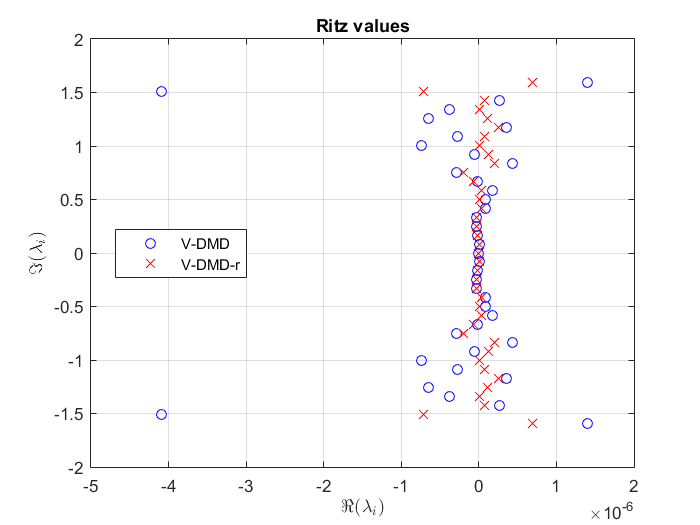}
					\includegraphics[width=2.90in, height=3in]{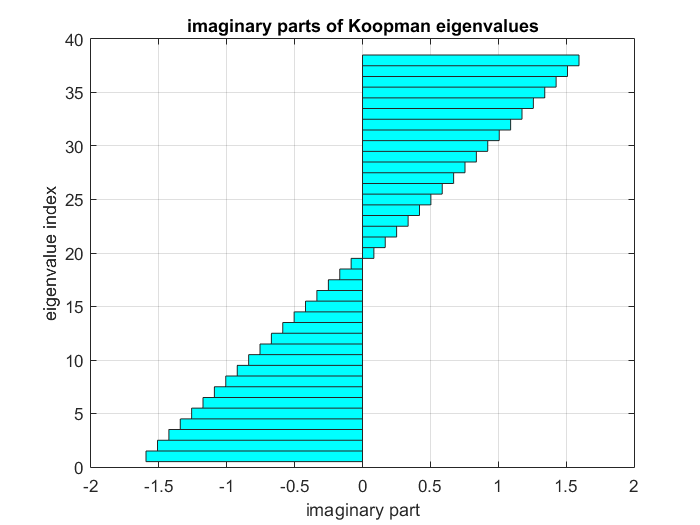} 
	\end{center}	
	\caption{\label{zd:fig:Ex1:Koop-eigs0-cylinder} (Example in \S \ref{zd:SS=Example-cylnder}) Left panel: Approximate Koopman eigenvalues computed by (\ref{koop-eig}) and Algorithm \ref{zd:ALG:DMD:RRR} with velocities as observables (circles \textcolor{blue}{$\circ$}) and with the additional refinement (crosses, \textcolor{red}{$\times$}). The refined Ritz values are defined as the Rayleigh quotients with the refined refined vectors; see {item} 3. in Theorem \ref{zd:TM:KRR-summary} and Remark \ref{zd:REM:rho(i)}. Here the movement closer to the imaginary axis by virtue of (\ref{koop-eig}) means getting closer to the unit disc in Figure \ref{zd:fig:Ex1:resids-cylinder}. Observe the scale on the real axis.  Right panel: The bar graph of the  imaginary parts of the Ritz values selected by the residual thresholding. The cyclic group structure of the eigenvalues on the unit circle is nicely visualized in the $\koop_j$'s; cf.  (\ref{koop-eig}). These values nicely correspond  to the analytically derived formulas for the Koopman eigenvalues of the flow \cite{bagheri-2013}.}
\end{figure} 


%

We conclude this experiment with an observation that only illustrates the issue of choosing properly scaled variables. This, together with the numerical linear algebra framework set in \S \ref{zd:S=General-weighted}, opens further important themes in numerical analysis of spectral properties of Koopman operator and it is the subject of our future research. 

\begin{figure}[H]
	\begin{center}
		\includegraphics[width=2.90in, height=2in]{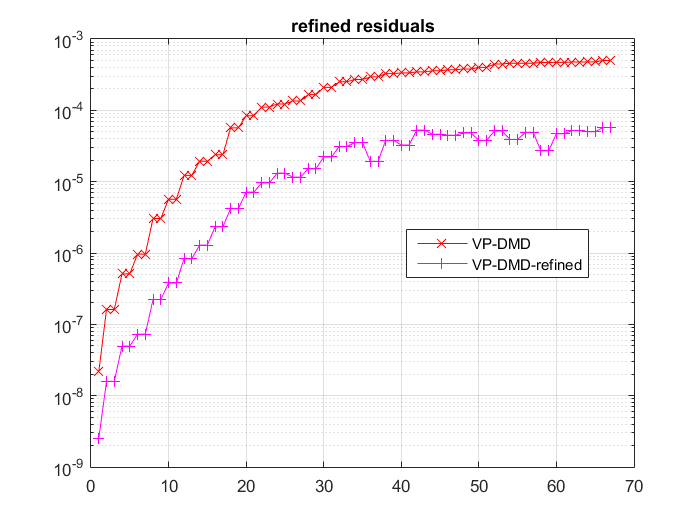}
		\includegraphics[width=2.90in, height=2in]{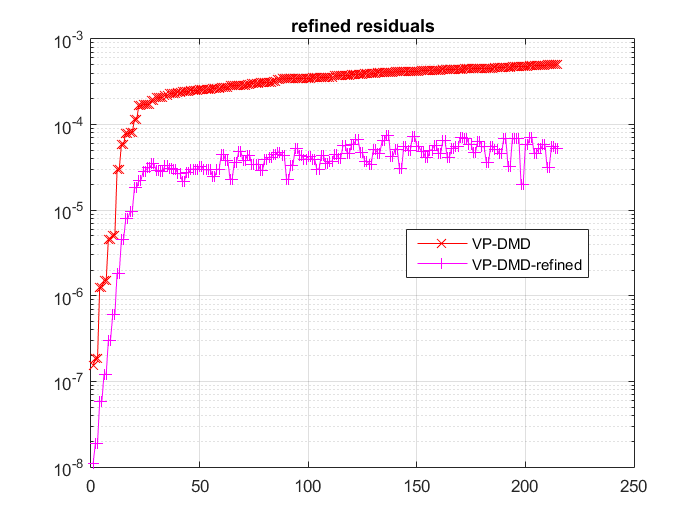}
	\end{center}	
	\caption{\label{zd:fig:Ex1:refined-resids-cylinder-Pxair} (Example in \S \ref{zd:SS=Example-cylnder}) Comparison of the  residuals of the Ritz pairs computed by Algorithm \ref{zd:ALG:DMD:RRR} with velocities and pressure as observables, but with differently scaled pressure values (i.e. to obtain the real pressure). The values of the corresponding refined residuals are also shown. On the left, the residuals of the top 67 pairs (that are below the residual threshold set to $5\cdot 10^{-4}$) are obtained if the scaling factor is $\rho=1.225$. On the right panel, the data for $215$ selected Ritz pairs corresponds to the scaling $\rho=1000$.  
}
\end{figure} 



%
%

\section{Implications to the Exact DMD}\label{zd:S=Exact-DMD}
In the Exact DMD  \cite[Algorithm 2]{Tu-DMD-Theory-Appl},  the action of the unaccessible operator $\A$ is mimicked by $\widehat{A} = \Y_m \X_m^{\dagger}$. Since the null-space of $\X_m$ is a subspace of the null space of $\Y_m=\A\X_m$, then $\widehat{A} \X_m = \Y_m$; for details we refer to \cite{Tu-DMD-Theory-Appl}. Based on the available information, we cannot distinguish $\restr{\A}{\mathcal{X}_m}$ from $\restr{\widehat{A}}{\mathcal{X}_m}$.
Further, Exact DMD cleverly uses the fact that $\mathcal{Y}_m$ is $\widehat{A}$--invariant ($\widehat{A}\Y_m = \Y_m(\X_m^\dagger \Y_m)$) and, thus, it contains exact eigenvectors of $\widehat{A}$.   
%
After computing the Rayleigh quotient matrix $S_k$ and its eigenpairs $\Lambda_k$, $W_k$ ($S_k W_k = W_k\Lambda_k$) in the same way as the DMD does, instead of the Ritz vectors $Z_k = U_k W_k$, the approximate eigenvectors are computed as $Z_k = \Y_m (V_k\Sigma_k^{-1} W_k)\Lambda_k^{-1}$.
It is easily checked that $Z_k = \A U_k W_k \Lambda_k^{-1}$.  

As in the case of DMD, the details of the tolerances for the truncated SVD used in Exact DMD are lacking in the literature. Here we suggest, following the discussion in \S \ref{SSS=Choose-k}, to apply the same strategy as in Algorithm \ref{zd:ALG:DMD:RRR}.
It should be noted here that the definition of $\widehat{A}$ is invariant under the scaling discussed in \S \ref{SS=Scaling}.
Indeed, $(\Y_m D)(\X_m D)^\dagger = \Y_m \X_m^\dagger$ for any $m\times m$ nonsingular matrix $D$.  
However, the numerical rank and the truncated SVD of $\X_m$ will heavily depend on the scaling. In fact, since the Exact DMD and DMD compute the same Ritz eigenvalues, repeating the experiment from \S \ref{zd:SSS=Synth-Ex} would reveal the same problem. 
%
On the other hand, if we apply scaling, the results are as in Figure \ref{zd:fig:Ex1:EDMD}. This clearly demonstrates that Exact DMD as well can benefit  from careful scaling of the data.\footnote{Here we reiterate the comment from the end of \S \ref{SS=Scaling}.} Smaller residuals shown of the right plot are due to the refinement of the Ritz vectors.

\begin{figure}[H]
	\begin{center}
		\includegraphics[width=\linewidth, height=2.2in]{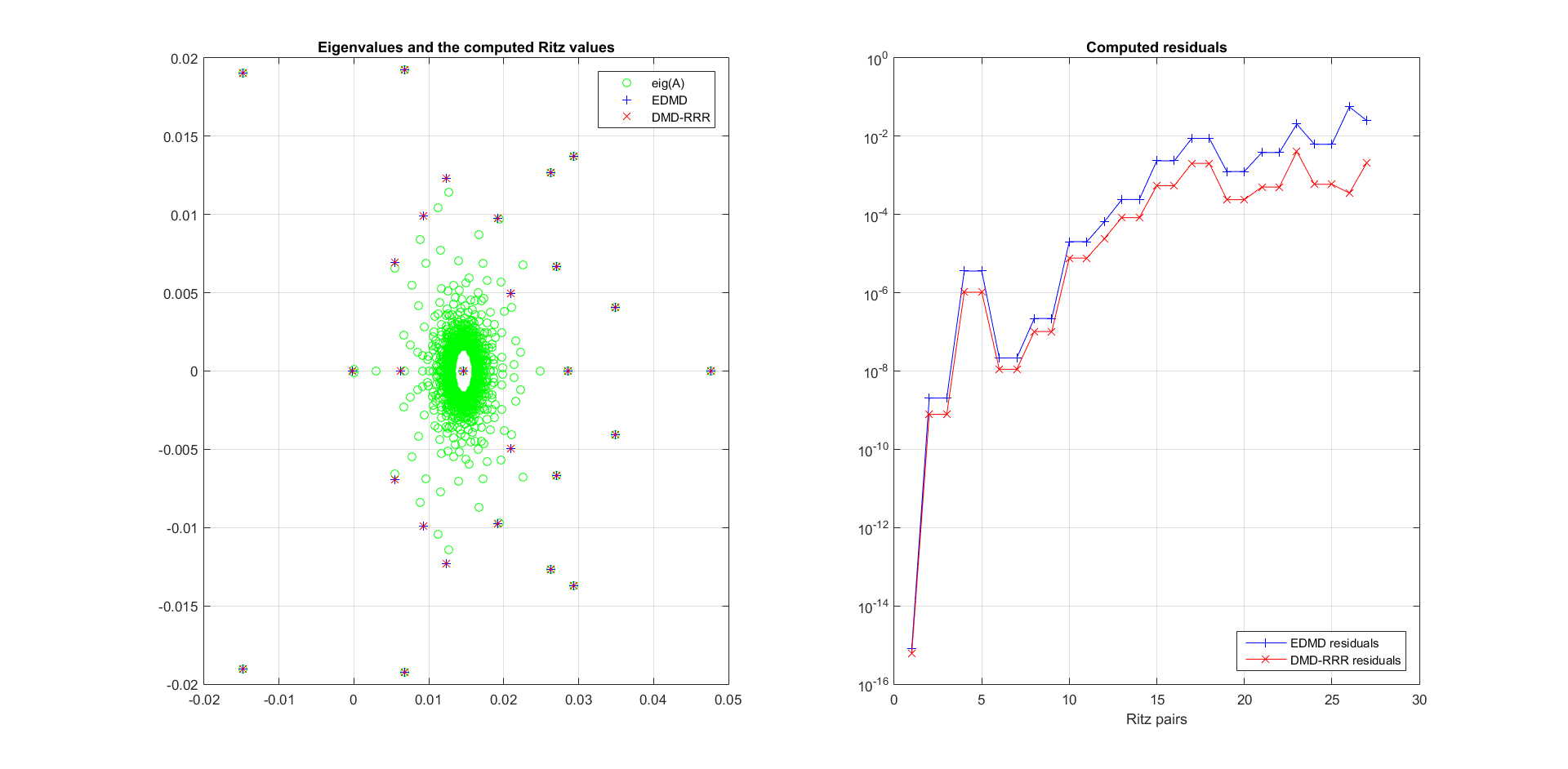}
	\end{center}	
	\caption{\label{zd:fig:Ex1:EDMD} (Example in \S \ref{zd:SSS=Synth-Ex-1.2}, with sequential data) Comparison of the residuals of the Ritz pairs computed by scaling enhanced Exact DMD (pluses $+$) and Algorithm \ref{zd:ALG:DMD:RRR} (crosses, $\times$).}
\end{figure}

{Unfortunately, we cannot compute the residuals of the Ritz pairs with the so called \emph{``exact eigenvectors''} $Z_k = \Y_m (V_k\Sigma_k^{-1} W_k)\Lambda_k^{-1}$.}
Namely, if we use all snapshots, we cannot test the quality of the Ritz vectors as we do not know how to apply $\A$ on them\footnote{In the case of sequential data, we can decide not to use $\f_{m+1}$ in the DMD construction, so we can use it for computing the residuals. In that case, the Exact DMD works with the subspaces of dimension $m-1$.} -- the data does not contain enough information for computing $\A \Y_m$, which is needed to compute $\A Z_k(:,i)$. In the case of sequential data (\ref{zd:eq:XY}), the convergence mechanism of power iterations (see Remark \ref{zd:Remark-POWIT}) may help.
\begin{proposition} Let $\X_m$ and $\Y_m$ be as in (\ref{zd:eq:XY}), and let $y = \sum_{i=1}^m \eta_i \f_{i+1} \in\Y_m$. Then
	\begin{equation}\label{zd:eq:A-A}
	\A y = \widehat{A} y + \eta_m \A (I-\X_m \X_m^\dagger)\f_{m+1} \equiv \widehat{A}y + \eta_m \A r_{m+1},
	\end{equation}
where $r_{m+1} = \f_{m+1}-\X_m (\X_m^\dagger \f_{m+1})$ is the optimal residual as in \S \ref{SSS=Krylov-decomposition}. In particular, if $y$ is an eigenvector with $\widehat{A}y=\lambda y$, then $\A y - \lambda y = \eta_m \A r_{m+1}$.
\end{proposition}

Hence, if $\f_1$ is well chosen and $m$ is big enough, then $r_{m+1}$ will be small and (\ref{zd:eq:A-A}) becomes useful, {provided that $\A$ does not amplify the residual}. For general data, such that we only have $\Y_m = \A \X_m$, not much can be said on $\A \Y_m$.


\section{Memory efficient DMD}\label{zd:S=Compressed-DMD} 
 In the case of high resolution numerical simulations the ambient space dimension $n$ (determined by the fineness of the grid) can be in millions, and mere memory fetching and storing is the bottleneck that precludes efficient computation, let alone flop count. 
 For instance, \cite{Sayadi-Schmid-2014}, \cite{Sayadi-Schmid-2016} report computation with the dimension $n$ of half billion on a massively parallel computer. 
  Tu and Rowley \cite{Tu20125317} discuss an implementation of DMD that has been designed to reduce computational work and memory requirements; this approach is also adopted in the library \texttt{modred} \cite{Belson:2014:ACMTOMS-945}.
 Sayadi and Schmid \cite{Sayadi-Schmid-2014} propose a parallelized SVD of $\X_m$, based on the QR factorization,  which improves the performance of the whole DMD since the SVD is its most expensive part. 
 
 In fact, it is a well known technique to preprocess the SVD with a QR factorization,  $\X_m=Q_x\left(\begin{smallmatrix} R_x \cr 0\end{smallmatrix}\right)$,
 and then to compute the SVD of the $m\times m$ upper triangular matrix $R_x$, and finally to assemble the left singular vectors using the unitary factor $Q_x$; see \cite{chan-1982}. The optimal ratio $n/m$ for this strategy depends on a concrete SVD algorithm and on a computing hardware and software. For example,  in the LAPACK \cite{LAPACK} library, the \texttt{xGESVD} and \texttt{xGESDD} subroutines\footnote{Following the LAPACK naming convention, in the subroutine name '\texttt{xGESVD}', \texttt{x} is one of \texttt{S, D, C, Z} for the four data types.} for computing the SVD with the bidiagonalization based QR SVD and the divide and conquer methods, the crossover point is obtained by calling e.g. \texttt{ILAENV( 6, 'xGESVD', JOBU // JOBVT, M, N, 0, 0 )}. The LAPACK's Jacobi SVD subroutines \texttt{xGEJSV} start with the QR factorization as well. 
 In the case $n\gg m$, the advantage is obvious and in particular because the QR factorization of a tall and skinny matrix can be optimized for various computing platforms \cite{Demmel-QR-LU-2012}.
 
 Hence, the modification proposed  in \cite{Sayadi-Schmid-2014} is already contained in the LAPACK SVD subroutines.

\subsection{QR Compressed DMD}\label{zd:SS=Compressed-DMD} 
On the other hand, all action contained in the sequential data\footnote{We first consider sequential data and later generalize to general $\X_m$ and $\Y_m=\A \X_m$.} in (\ref{zd:eq:XY}) at step $i=m$ is confined to at most $m+1$ dimensional range of $\F_{m+1}=(\f_1,\ldots, \f_{m+1})$. Hence, all previously analyzed algorithms can be represented in $\mathrm{range}(\F_{m+1})$, and it only remains to construct a convenient orthonormal basis. To that end, let 
\begin{equation}\label{zd:eq:QRFm+1}
\F_{m+1} = Q_f \begin{pmatrix} R_f \cr 0\end{pmatrix} = \widehat{Q}_f R_f,\;\;Q_f^* Q_f = I_n, 
\end{equation}
be the QR factorization of $\F_{m+1}$: $\widehat{Q}_f=Q_f(:,1:m+1)$, $\mathrm{range}(\widehat{Q}_f)\supseteq \mathrm{range}(\F_{m+1})$, and $R_f$ is $(m+1)\times (m+1)$ upper triangular. (If $\F_{m+1}$ is of full column rank, then $\mathrm{range}(\widehat{Q}_f)= \mathrm{range}(\F_{m+1})$.)

If we set $R_x=R_f(:,1:m)$, $R_y=R_f(:,2:m+1)$, then 
\begin{equation}\label{zd:eq:FXY-compress}
\X_m = \widehat{Q}_f R_x,\;\;\Y_m =\widehat{Q}_f R_y \; ;\;\;
R_f = \left( \begin{smallmatrix} \times & \divideontimes & \divideontimes & \divideontimes & \div \cr & \divideontimes & \divideontimes & \divideontimes & \div \cr 
 &  & \divideontimes & \divideontimes & \div \cr 
 &   &  & \divideontimes & \div \cr
 &   &   &   & \div\end{smallmatrix}\right),\;\;
 R_x = \left( \begin{smallmatrix} 
 \times & \divideontimes & \divideontimes & \divideontimes \cr 
        & \divideontimes & \divideontimes & \divideontimes \cr
        &        & \divideontimes & \divideontimes \cr
        &        &        & \divideontimes \cr
        &        &        &   0
 \end{smallmatrix}\right),\;\;
 R_y = \left(\begin{smallmatrix} 
 \divideontimes & \divideontimes & \divideontimes & \div \cr
 \divideontimes & \divideontimes & \divideontimes & \div \cr
      & \divideontimes & \divideontimes & \div \cr
      &      & \divideontimes & \div \cr
      &      &      & \div
  \end{smallmatrix}\right) ,
\end{equation}
and, in the basis of the columns of $\widehat{Q}_f$, we can identify $\X_m\equiv R_x$, $\Y_m\equiv R_y$, i.e. we can think of $\X_m$ and $\Y_m$ as $m$ snapshots in an $(m+1)$ dimensional space. 
Recall that $m\ll n$.
 
If we apply the DMD algorithm to this data, it will return the matrix of approximate eigenvalues 
$\Lambda_k$ with the corresponding eigenvectors as the columns of $\widehat{Z}_k\in\mathbb{C}^{(m+1)\times k}$. 
To transform this output in terms of the original data, it suffices to lift the eigenvectors as $Z_k = \widehat{Q}_f \widehat{Z}_k$.
Also note that this change of basis is actually an implicit unitary similarity applied to $\A$, since
\begin{equation}\label{eq:QfTAQf}
\X_m = Q_f \begin{pmatrix} R_x \cr 0\end{pmatrix},\;\;
\Y_m = \A  Q_f \begin{pmatrix} R_x \cr 0\end{pmatrix}= Q_f \begin{pmatrix} R_y \cr 0\end{pmatrix} \; \Longrightarrow\; \begin{pmatrix} R_y \cr 0 \end{pmatrix} = (Q_f^* \A Q_f) \begin{pmatrix} R_x \cr 0\end{pmatrix} .
\end{equation}
Note that $R_y = (\widehat{Q}_f^*\A \widehat{Q}_f) R_x$, and, analogously, in the invertible case, $R_x = (\widehat{Q}_f^*\A^{-1} \widehat{Q}_f) R_y$ . 
In the case of the exact DMD with $\widehat{A}=\Y_m\X_m^{\dagger}$ we have
\begin{equation}
\X_m^{\dagger} = \begin{pmatrix} R_x^{\dagger} & 0 \end{pmatrix} Q_f^* = R_x^{\dagger} \widehat{Q}_f^*,\;\;
\widehat{A} = \Y_m\X_m^{\dagger} = Q_f \begin{pmatrix} R_y R_x^{\dagger} & 0 \cr 0 & 0 \end{pmatrix} Q_f^* = 
\widehat{Q}_f R_y R_x^{\dagger} \widehat{Q}_f^* .
\end{equation}   

The benefit of this compressed DMD is that the dimension $n$ is reduced to much smaller dimension $m+1$ using the QR factorization that can be optimized for tall and skinny matrices \cite{Demmel-QR-LU-2012}. 

The cost of this reduction is comparable  to the initial QR factorization of $\X_m$ proposed in  \cite{Sayadi-Schmid-2014}. However, unlike  \cite{Sayadi-Schmid-2014}, where this is considered only a preprocessing step toward more efficient SVD of $\X_m$ in DMD  running in the original $n$-dimensional space, our compressed DMD runs entirely in the $(m+1)$ dimensional subspace of $\mathbb{C}^n$, with the nice interpretation (\ref{eq:QfTAQf}). 

This compression trick applies to Algorithm \ref{zd:ALG:DMD}, Algorithm \ref{zd:ALG:DMD:RRR} as well as to the Exact DMD. It especially greatly improves the efficiency of Algorithm \ref{zd:ALG:DMD:RRR} because it reduces the overhead of computing the refined Ritz vectors.
For the readers convenience, in Algorithm \ref{zd:ALG:DMD:RRR:compressed}, we show the compressed version of Algorithm \ref{zd:ALG:DMD:RRR}; for the other two algorithms, \emph{mutatis mutandis}, the corresponding compressed versions are straightforward.

\begin{algorithm}[hbt]
	\caption{$[Z_k, \Lambda_k, r_k, \rho_k]=\mathrm{DDMD\_RRR\_C}(\F_m; \epsilon)$ \{\emph{QR Compressed Refined DDMD\_RRR}\}}
	\label{zd:ALG:DMD:RRR:compressed}
	\begin{algorithmic}[1]
		\REQUIRE \  \\		
		\begin{itemize} 
			\item $\F_{{m+1}}=(\f_1,\ldots,\f_m,\f_{m+1})$ that defines a sequence of snapshots pairs $\f_{i+1}=\A \f_i$. (Tacit assumption is that $n$ is large and that $m \ll n$.)
			\item Tolerance level $\epsilon$ for numerical rank determination.		
		\end{itemize}
		\STATE $[\widehat{Q}_f,R_f]=qr(\F_{m+1},0)$ ; \COMMENT{\emph{thin QR factorization}}
		\STATE  $R_x=R_f(1:m+1,1:m)$, $R_y=R_f(1:m+1,2:m+1)$ ; \COMMENT{\emph{New representaitons of $\X_m$, $\Y_m$.}}
        \STATE $[\widehat{Z}_k, \Lambda_k, r_k, \rho_k]=\mathrm{DDMD\_RRR}(R_x, R_y; \epsilon)$; \COMMENT{\emph{Algorithm \ref{zd:ALG:DMD:RRR} in $(m+1)$-dimensional ambient space}}
		\STATE $Z_k = \widehat{Q}_f \widehat{Z}_k$ 
		\ENSURE $Z_k$, $\Lambda_k$, $r_k$, $\rho_k$
	\end{algorithmic}
\end{algorithm}

\begin{remark}
In an efficient software implementation, the matrix $\widehat{Q}_f$ is not formed explicitly. Instead, the information on the $m+1$ Householder reflectors used in the factorization is stored in the positions of the annihilated entries (see e.g. \texttt{xGEQRF} in LAPACK) and then one can apply such implicitly stored $Q_f$ and compute 	$Z_k = Q_f \left(\begin{smallmatrix}\widehat{Z}_k \cr 0\end{smallmatrix}\right)\equiv \widehat{Q}_f \widehat{Z}_k $ (see e.g. \texttt{xORMQR} in LAPACK).	
\end{remark}

\begin{remark}
In the case of general data $\X_m$, $\Y_m=\A \X_m$, instead of the QR factorization (\ref{zd:eq:QRFm+1}), we can compress the data onto the $2m$ dimensional subspace using the QR factorization
$$
\begin{pmatrix} \X_m & \Y_m\end{pmatrix} = Q_{xy} \begin{pmatrix} R_{xy}\cr 0\end{pmatrix}  = \widehat{Q}_{xy} R_{xy}.
$$
and analogously to (\ref{zd:eq:FXY-compress}), the low dimensional representations $\X_m = \widehat{Q}_{xy} R_x$, $\Y_m=\widehat{Q}_{xy} R_y$, where $\widehat{Q}_{xy}=Q_{xy}(:,1:2m)$, $R_x = R_{xy}(1:2m,1:m)$, $R_y=R_{xy}(1:2m,m+1:2m)$. 
As in the case of sequential data (see Algorithm \ref{zd:ALG:DMD:RRR:compressed}), the DMD will compute with $R_x$ and $R_y$, and the Ritz vectors will be transformed back to $\mathbb{C}^n$ using $Q_{xy}$. 
\end{remark}

\begin{remark}
		{Using (\ref{zd:eq:QRFm+1}), (\ref{zd:eq:FXY-compress}) as in Algorithm \ref{zd:ALG:DMD:RRR:compressed} facilitates efficient updating/down-dating if we keep adding new snapshots and/or dropping  the ones at the beginning,  e.g. if the snapshots are taken from a sliding (in discrete time steps) window that may even be of variable width. Also, rows (observables) may be added/removed and the decomposition recomputed from the previous one. The key is that only the QR factorization (\ref{zd:eq:QRFm+1}) needs to be updated, using the well established algorithms (see e.g. \cite{Hammarling-Lucas-2008}) that are also available for parallel computing (see e.g. \cite{ANDREW2014161}); the rest of the computation takes place in $(m+1)$-dimensional space. We omit the details for the sake of brevity.}		
\end{remark}

\section{Forward--Backward DMD (F-B DMD)}\label{S=FB-DMD}
Dawson et al. \cite{Dawson2016} (see also \cite[\S 8.3]{dmd-book-kutz-2016}) proposed a Forward--Backward version of DMD designed to reduce the bias caused by sensor noise. The idea is to run (the exact) DMD also backward by swapping $\X_m$ and $\Y_m$, thus working implicitly with $\A^{-1}$.
If the corresponding Rayleigh quotients are $S_k$ and $S_{k,back}$ (for backward DMD) then, assuming nonsingularity of both Rayleigh quotients, the projected operator is approximated by 
$$
\widehat{S}_k=\sqrt{S_k S_{k,back}^{-1}}.
$$ 
The intuition is that the bias in the eigenvalues of $S_k$ and $S_{k,back}$ will be subject to cancellation in the product $S_k S_{k,back}^{-1}$.
Now, if $\widehat{S}_k w_i = \mu_i w_i$, then for every nonzero $\mu_i$, the corresponding approximate eigenvector is $U_r w_i$ (standard DMD) or 
$\mu_i^{-1}\Y_m V_r\Sigma_r^{-1}w_i$ (Exact DMD, preferred in \cite[Algorithm 1, Algorithm 3]{Dawson2016}). Numerical examples provided in  \cite{Dawson2016} and \cite[\S 8.3]{dmd-book-kutz-2016} demonstrate the efficacy of the F-B DMD. However, numerical details of the effective implementation of the method remain open for discussion with a potential for improvement, in particular with respect to the usage of the matrix square root. 


In \cite{Dawson2016}, the problem of computing the matrix square root is discussed in detail, in particular with respect to the non-uniqueness\footnote{A nonsingular matrix $M$ with $d$ distinct eigenvalues has $2^d$ square roots which are matrix functions of $M$.} and it is proposed to resolve the ambiguity by choosing the root closest to $S_k$ (or $S_{k,back}^{-1}$). This may be difficult to implement in practice, as computing the matrix square root is a nontrivial task, and this additional constraint would make it even more difficult. Furthermore, in a software implementation of the DMD, one usually resorts to state of the art software packages, such as e.g. Matlab, and uses black-box routine. For instance,  the function \texttt{sqrtm()} (Matlab) computes the branch of the square root such that all eigenvalues have nonnegative real parts; each real negative eigenvalue of $S_k S_{k,back}^{-1}$ will generally produce in the spectrum of $\sqrt{S_k S_{k,back}^{-1}}$ a  purely imaginary eigenvalue in the upper half plane.
Moreover, the computed square root is in general complex, even when 
$S_k S_{k,back}^{-1}$ is real. (It is known that real matrix possesses a real square root if and only if each of its elementary divisors with real negative eigenvalues  occurs an even number of times \cite[Theorem 5]{HIGHAM1987405}.)

In \cite[Algorithm 8.6]{dmd-book-kutz-2016}, the square root is computed in Matlab as the matrix fractional power, $(S_k S_{k,back}^{-1})^{0.5}$. This function is also numerically difficult way to compute the square root, see e.g. \cite{Higham-Lin-Power}.
For more systematic treatment of matrix square root function we refer to \cite[Chapter 6]{Higham:2008:FM}. 

\subsection{Matrix root free Forward--Backward DMD}
Closer look at the Forward--Backward DMD reveals that, intrinsically, we do not need $\widehat{S}_k=\sqrt{S_k S_{k,back}^{-1}}$, but (only) its spectral decomposition. Hence, the FB-DMD can be implemented without invoking the matrix square root; instead we deploy the spectral mapping theorem.
The result is the a Matrix Root--Free Forward--Backward DMD, outlined in Algorithm \ref{zd:ALG:DMD:F-B-MRF}:
\begin{algorithm}[hbt]
	\caption{$[Z_k, \Lambda_k, r_k]=\mathrm{DMD\_FB\_MRF}(\F_m; \epsilon)$ \{\emph{Matrix--Root--Free Forward--Backward DMD}\}}
	\label{zd:ALG:DMD:F-B-MRF}
	\begin{algorithmic}[1]
		\REQUIRE \  \\		
		\begin{itemize} 
			\item $\X_m=(\x_1,\ldots,\x_m), \Y_m=(\y_1,\ldots,\y_m)\in \mathbb{C}^{n\times m}$ that define a sequence of snapshots pairs $(\x_i,\y_i\equiv \A \x_i)$. (Tacit assumption is that $n$ is large and that $m \ll n$.)
			\item Tolerance level $\epsilon$ for numerical rank determination.		
		\end{itemize}	
		\STATE Apply a DMD algorithm to $(\X_m, \Y_m\equiv\A \X_m)$ and compute the POD basis $U_k$ and the Rayleigh quotient $S_k$.  Use a threshold strategy to determine the reduced dimension $k$.
		\STATE Apply the same DMD scheme to $(\Y_m,\X_m\equiv \A^{-1}\Y_m)$, but with fixed dimension $k$ of the POD basis and compute the Rayleigh quotient $S_{k,back}$.
		\STATE $M_k = S_k S_{k,back}^{-1}$ ; \COMMENT{\emph{Without explicit inverse, e.g. $M_k = S_k\slash S_{k,back}$ (Matlab).}}
		\STATE $[W_k, \Omega_k] = \mathrm{eig}(M_k)$ \COMMENT{$\Omega_k=\mathrm{diag}(\omega_i)_{i=1}^k$; $M_k W_k(:,i)=\omega_i W_k(:,i)$; $\|W_k(:,i)\|_2=1$}
		\STATE $\lambda_i = \pm \sqrt{\omega_i}$, $\Lambda_k = \mathrm{diag}(\lambda_i)_{i=1}^k$. \COMMENT{\emph{Choose the signs carefully; see Remark \ref{zd:REM:sign-lambda}.}}
		\STATE $Z_k = U_k W_k$ \COMMENT{\emph{Ritz vectors. Alternatively, use the vectors from the Exact DMD as in \cite{Dawson2016}.}}
		\STATE $r_k(i) = 
		\| (\Y_m V_k\Sigma_k^{-1}) W_k(:,i) - \lambda_i Z_k(:,i) \|_2$, $i=1,\ldots, k$.
		\ENSURE $Z_k$, $\Lambda_k$, $r_k$
	\end{algorithmic}
\end{algorithm}
 
\begin{remark}\label{zd:REM:sign-lambda}
In Line 5., the complex square root function computes the principal value, $\sqrt{\omega_i}=\sqrt{(|\omega_i|+\Re(\omega_i))/2} \pm \ii \sqrt{(|\omega_i|-\Re(\omega_i))/2}$. Hence $\Re(\sqrt{\omega_i})\geq 0$, $i=1,\ldots, k$. If $M_k$ is real, its complex (non-real) eigenvalues occur in complex conjugate pairs and their square roots will inherit closedness under conjugation. (In Matlab, the sign of the imaginary part of $\sqrt{\omega_i}$ matches the sign of the imaginary part of $\omega_i$. For more details on machine implementation of complex functions see \cite{Kahan-branch-cuts}.) Note however that any real negative $\omega_i$ yields purely imaginary $\lambda_i = \ii \sqrt{-\omega_i}$, and the conjugacy symmetry will be lost.		
One way to choose the sign in the definition $\lambda_i=\pm \sqrt{\omega_i}$ is e.g. by comparison with 
$W_k(:,i)^* S_k W_k(:,i)$, or with the Ritz values of $S_k$.
\end{remark}

\begin{remark}
For more efficient computation, \cite{Dawson2016} proposes optional projection of both $\X_m$ and $\Y_m$ onto the POD basis for $\X_m$, so that the F-B DMD runs in the POD basis.		
An alternative for both the original FB-DMD and Algorithm \ref{zd:ALG:DMD:F-B-MRF} is to use the QR Compressed version of the DMD (\S \ref{zd:SS=Compressed-DMD}), and thus to work with $R_x$ and $R_y$ instead of $\X_m$ and $\Y_m$, respectively. This is an attractive option because of availability of efficient software implementations of the QR factorization of tall and skinny matrices, see e.g. \cite{Demmel-QR-LU-2012}.  Since this corresponds to an orthogonal (or unitary) transformation of the data, the model for the noise in the derivation of the FB-DMD remains valid.
\end{remark}

\section{General framework: weighted DMD}\label{zd:S=General-weighted}
The truncated SVD provides optimal low rank approximation to $\X_m$ by favoring most energetic modes, all in the Euclidean inner product structure that is usually taken as default. This may not be optimal and restricting the Rayleigh-Ritz procedure to only a subspace of $\mathcal{X}_m$ may prevent it from finding low energy but relevant modes. Further, truncating the SVD of $\X_m$ in the canonical Euclidean structure (and deeming the vectors large or small in the Euclidean norm) ignores that the numerical values in a snapshot $\f_i$ represent physical variables such as pressure and velocity components, each in its own units.\footnote{See e.g. \cite{Richecoeur-DMD-combustion-2012}, where separate DMD is computed for each variable.} Clearly, care must be  exercised in computing and interpreting norm of such a vector. Also, if the components of $\f_i$ are measured values contaminated with noise, then statistical information (covariance) may be available, and must be taken into account.

Moreover, dimension reduction by a Galerkin/POD projection in the framework of an inappropriate Hilbert space structure may cause loss of important physical properties (e.g. stability) and some other important principles (e.g. conservation) may get lost in the projection.  
In such cases, it may be crucial to use an appropriate notion of energy i.e. the inner product that is in the discretized space defined by a positive definite matrix $M$ ($(x,y)_M = y^* M x$); recall also that discretization of a continuous inner product defined by an integral yields a weighted inner product defined through the appropriate quadrature weights. For an example of construction of $M$ see e.g. \cite{Satish-et-all-Orr-Sommerfeld}. 
In some cases, the natural inner product structure is defined by the positive definite solution $M$ of a Lyapunov matrix equation, see e.g. \cite[\S 6.1]{ROM-SANDIA-2014}, \cite{Serre20125176}, \cite{Rowley-MRF-2005}.
As an illustration of the importance of the first step we refer to \cite{POD-Sound-Freund}, \cite{POD-Symm-TW}, \cite{Kalashnikova-Arun-2014} for examples how different inner products yield completely different results -- certainly not desirable situation for real world applications. 

If a weighted Galerkin/POD projection is combined with other approximation techniques such as the DEIM
\cite{grepl-maday-2007}, \cite{EIM}, \cite{Maday2009383}, \cite{DEIM} or the QDEIM \cite{drmac-gugercin-DEIM-2016}, then it is important that the overall computation and the error bounds are done in the same (weighted) inner product; for a recent development of this idea in the context of DEIM see \cite{GEIM}, \cite{Drmac-Saibaba-WDEIM}.

Hence, it seems natural and important to formulate the first step in the Schmid's DMD algorithm (as well as in other versions of the DMD) in a general weighted inner product space.

To set the scene, we recall the weighted POD \cite{volkwein-2011-mor}, summarized in Algorithm \ref{zd:ALG:POD}. 
\begin{algorithm}[H]
	\caption{$[\widetilde{U}_k, L, [\widehat{U}_k]]=\mathrm{Weighted\_POD}(\X_m; M; \epsilon)$ (See e.g. \cite{volkwein-2011-mor}.)}
	\label{zd:ALG:POD}
	\begin{algorithmic}[1]
		\REQUIRE \  \\		
		\begin{itemize} 
			\item The matrix of snapshots $\X_m=\begin{pmatrix}\f_1, \f_2, \ldots, \f_m\end{pmatrix} \in\mathbb{C}^{n\times m}$. (Assumed $m < n$.)
			\item Hermitian $n\times n$ positive definite $M$ that defines the inner product.	
			\item Tolerance level $\epsilon$ for numerical rank detection.		
		\end{itemize}
		\STATE Compute the Cholesky decomposition $M = LL^*$. \COMMENT{\emph{Use the definition of $M$ and exploit structure. $M$ can also be diagonal. $M$ can also be given implicitly by providing $L$, which is not necessarily triangular. One can take $L=\sqrt{M}$.}}
		\STATE Compute the thin SVD  $L^* \X_m = \widetilde{U} \Sigma V^*$. \COMMENT{$\Sigma=\mathrm{diag}(\sigma_i)_{i=1}^m$; $\sigma_{1}\geq\cdots\geq\sigma_m$.}
		\STATE Determine the numerical rank $k$ using the threshold $\epsilon$. \COMMENT{\emph{See \S \ref{SSS=Choose-k}.}}
		\ENSURE $\widehat{U}_k \equiv L^{-*} \widetilde{U}_k$, where $\widetilde{U}_k=\widetilde{U}(:,1:k)$. $\widehat{U}_k$ can be returned implicitly by $L$ and $\widetilde{U}_k$.
	\end{algorithmic}
\end{algorithm}
\noindent The computed weighted POD basis $\widehat{U}_k$ is $M$--unitary: $\widehat{U}_k^* M  \widehat{U}_k = I_k$, i.e. its columns are orthonormal set in the inner product $(x,y)_M = y^* Mx$. Any computation built upon the result of Algorithm \ref{zd:ALG:POD} (including the Rayleigh quotient, Ritz pairs and error estimates) should be done in the space $(\mathbb{C}^n,(\cdot,\cdot)_M)$, and the errors should be measured in the corresponding norm $\|\cdot\|_M = \sqrt{(\cdot,\cdot)_M}\equiv \| L^* \cdot\|_2$,  where $L$ is any square matrix such that $LL^*=M$; for an example see e.g.  \cite{Drmac-Saibaba-WDEIM}.
%
Recall that the operator norm induced by the norm $\|\cdot\|_M$ can be computed as $\|\A\|_M = \| L^* \A L^{-*}\|_2$,

\begin{remark}
The optimality property of the low rank approximation $\X_m \approx \widehat{U}_k \Sigma_k V_k^*$ (where $\Sigma_k=\Sigma(1:k,1:k)$, $V_k=V(:,1:k)$) still holds true, but measured in the induced matrix norm 
$\|\X_m\|_{2,M}=\max_{z\neq 0} \|\X_m z\|_M /\|z\|_2 \equiv \|L^* \X_m\|_2 .$ It is easily checked that $\|\X_m\|_{2,M}=\sigma_1$, and that the approximation error $\Delta \X_{m,k}\equiv \X_m - \widehat{U}_k \Sigma_k V_k^*$ is equal to
$
\| \X_m - \widehat{U}_k \Sigma_k V_k^* \|_{2,M} = \sigma_{k+1} .
$
\end{remark}

\subsection{New algorithm}\label{SS=DMD-weighted-new-alg}
We now work out the details of the data driven Rayleigh--Ritz approximation in the weighted inner product.
\begin{proposition}\label{zd:PROP:RQ-weighted} 
Assume that we are given the data $\X_m$ and $\Y_m=\A \X_m$, and that Algorithm \ref{zd:ALG:POD} is applied to $\X_m$.
The $(\cdot,\cdot)_M$--weighted Rayleigh quotient matrix of $\A$ with respect to the POD basis $\widehat{U}_k$ can be computed from the available data as
	\begin{equation}\label{eq:hatSk}
	\widehat{S}_k = \widetilde{U}_k^* L^* \Y_m V_k\Sigma_k^{-1} .
	\end{equation}
If $(\lambda, w)$ is an eigenpair of $\widehat{S}_k$ ($\widehat{S}_k w=\lambda w$, $w\neq 0$) then the corresponding Ritz pair $(\lambda,\widehat{U}_k w)$ of $\A$ has the residual norm computable as
\begin{equation}\label{zd:eq:residual-weighted}
\| \A (\widehat{U}_k w) - \lambda (\widehat{U}_k w) \|_M = \| L^* \Y_m V_k\Sigma_k^{-1} w - \lambda \widetilde{U}_k w\|_2 . 
\end{equation}	
\end{proposition}
\begin{proof}
Recall that the adjoint of $\widehat{U}_k$ (considered as mapping between the inner product spaces  $(\mathbb{C}^k,(\cdot,\cdot))$ and $(\mathbb{C}^n,(\cdot,\cdot)_M)$) is $\widehat{U}_k^{[*]}=\widehat{U}_k^* M$, where $\widehat{U}_k^*$ is the usual conjugate transpose. Since $\widehat{U}_k^{[*]}\widehat{U}_k=I_k$, 
the Rayleigh quotient is computed as
\begin{equation}\label{zd:eq:Sk[*]}
\widehat{S}_k = \widehat{U}_k^{[*]} \A\widehat{U}_k = \widetilde{U}_k^* L^{*}\A L^{-*}\widetilde{U}_k .
\end{equation}
Finally, one can easily check that 
$$
\widetilde{U}_k^* L^* \Y_m V_k \Sigma_k^{-1} = \widetilde{U}_k^* L^*
\A (\widehat{U}_k \Sigma_k V_k^* + (\X_m - \widehat{U}_k \Sigma_k V_k^*) )V_k \Sigma_k^{-1} = \widetilde{U}_k^* L^{*}\A L^{-*}\widetilde{U}_k \equiv \widehat{S}_k ,
$$
{where we have used $\Delta\X_{m,k} V_k \Sigma_k^{-1} = 0$},
and  $\A \widehat{U}_k = \A (\X_m - \Delta\X_{m,k})V_k\Sigma_k^{-1} = \A \X_m V_k\Sigma_k^{-1} = \Y_m V_k\Sigma_k^{-1}$. 
This completes the proof of both (\ref{eq:hatSk}) and (\ref{zd:eq:residual-weighted}).
\end{proof}


\noindent The resulting weighted version of the DMD algorithm is as follows:\footnote{Note that this modification trivially applies to the Exact DMD \cite[Algorithm 2]{Tu-DMD-Theory-Appl}.}

\begin{algorithm}[H]
	\caption{$[Z_k,\Lambda_k] =\mathrm{Weighted\_DMD}(\X_m, \Y_m; M; \epsilon)$}
	\label{zd:ALG:DMD-weighted}
	\begin{algorithmic}[1]
		\REQUIRE \  \\		
		\begin{itemize} 
			\item $\X_m=(\x_1,\ldots,\x_m), \Y_m=(\y_1,\ldots,\y_m)\in \mathbb{C}^{n\times m}$ that define a sequence of snapshots pairs $(\x_i,\y_i\equiv \A \x_i)$. (Tacit assumption is that $n$ is large and that $m \ll n$.)
			\item Hermitian $n\times n$ positive definite $M$ that defines the inner product.	
			\item Tolerance level $\epsilon$ for numerical rank
		\end{itemize}
		\STATE $[\widetilde{U}_k, L, [\widehat{U}_k]]=\mathrm{Weighted\_POD}(\X_m; M; \epsilon)$ \COMMENT{\emph{Algorithm \ref{zd:ALG:POD}}}
		\STATE $\widehat{S}_k = \widetilde{U}_k^* L^* \Y_m V_k\Sigma_k^{-1}$ \COMMENT{\emph{Rayleigh quotient; see Proposition \ref{zd:PROP:RQ-weighted}.}}
		\STATE $[W_k, \Lambda_k] = \mathrm{eig}(\widehat{S}_k)$ \COMMENT{$\Lambda_k=\mathrm{diag}(\lambda_i)_{i=1}^k$; $\widehat{S}_k W_k(:,i)=\lambda_i W_k(:,i)$; $\|W_k(:,i)\|_2=1$}
		\STATE $Z_k = \widehat{U}_k W_k$ \COMMENT{\emph{Ritz vectors}}
		\ENSURE $Z_k$, $\Lambda_k$
	\end{algorithmic}
\end{algorithm}

\begin{remark}
In the case of sequentially shifted data, computing $L^* \Y_m$ in Line 2. of Algorithm \ref{zd:ALG:DMD-weighted} and in (\ref{zd:eq:residual-weighted}) (if residual norms are wanted as well)  can use the trailing $m-1$ columns of $L^* \X_m$ computed in Line 2. of Algorithm \ref{zd:ALG:POD}. This saves $O(n^2 m)$ floating point operations.
\end{remark}

\begin{remark}
	If, with given $M$, the natural structure is given in the inner product $(x,y)_{M^{-1}}\equiv y^* M^{-1}x$, then the above computation is modified as follows: \emph{(i)} In Line 2. of Algorithm \ref{zd:ALG:POD}, compute the SVD $L^{-1}\X_m = \widetilde{U}\Sigma V^*$; \emph{(ii)} Define $\widehat{U}_k$ as 
	$\widehat{U}_k = L\widetilde{U}_k$ ; \emph{(iii)} In Line 2. of Algorithm \ref{zd:ALG:DMD-weighted} define the Rayleigh quotient as $\widehat{S}_k = \widetilde{U}_k^* L^{-1}\Y_m V_k\Sigma_k^{-1}\; (\equiv \widetilde{U}_k^* L^{-1}\A L \widetilde{U}_k)$; \emph{(iv)} for a Ritz pair $(\lambda,\widehat{U}_k w)$ compute the residual as
	$\| \A (\widehat{U}_k w) - \lambda (\widehat{U}_k w) \|_{M^{-1}} = \| L^{-1} \Y_m V_k\Sigma_k^{-1} w - \lambda \widetilde{U}_k w\|_2$.
	These changes simply reflect the factorization $M^{-1}=L^{-*}L^{-1}$, which is equivalent to $M=LL^*$ (for square $L$).	
\end{remark}
		
\begin{remark}\label{zd:REM:A:M-unitary}
There is another important aspect of allowing more general inner product in the DMD. Think of $\A$ as being a discretized unitary\footnote{or, more generally, normal} operator in a Hilbert space with the inner product $(\cdot,\cdot)_M$, i.e. $\A$ is $M$-unitary: $\A^* M \A = M\equiv LL^*$.	Recall that the $M$-adjoint of $\A$ reads $\A^{[*]_M}=M^{-1}\A^* M$, and that $L^* \A L^{-*}$ (appearing in (\ref{zd:eq:Sk[*]})) is unitary in $(\cdot,\cdot)$. Hence, the weighted formulation automatically includes such general form of unitarity. Of course, all aspects of the more general setting must be adjusted for detailed error and perturbation analysis. For the sake of brevity, we only illustrate this point by providing below a weighted version of the Bauer--Fike theorem.
\end{remark}


\begin{theorem}\label{zd:TM:BF-weighted}
Let $\A = S \mathrm{diag}(\alpha_i)_{i=1}^n S^{-1}$, and let $\lambda$ be an eigenvalue of $\A + \delta\A$ with some perturbation $\delta\A$.
Then
\begin{equation}
\min_{i=1:n} |\alpha_i - \lambda| \leq \sqrt{\mu_2(M)} (\|S\|_M \|S^{-1}\|_M) \|\delta\A\|_M,\;\;\mbox{where}\;\; \mu_2(M)=\min_{{\Delta \in \mathrm{Diag}^+_n}}\kappa_2(\Delta M \Delta ) ,
\end{equation}	
{where $\mathrm{Diag}^+_n$ denotes the set of diagonal, positive definite matrices on $\C^n$.}
If $\A$ is $M$-normal, then $S$ is $M$-unitary and $\kappa_M(S)\equiv \|S\|_M \|S^{-1}\|_M=1$. If $M$ is diagonal, $\mu_2(M)=1$. Further, if $\A$ is nonsingular then
\begin{equation}
\min_{i=1:n} \frac{|\alpha_i - \lambda|}{|\alpha_i|} \leq \sqrt{\mu_2(M)} \kappa_M(S) \|\A^{-1}\delta\A\|_M.
\end{equation}	
\end{theorem}
\begin{proof}
As in the classical proof of the Bauer--Fike theorem (see {\cite[Theorem 1.6]{ste-sun-90}}), 
we conclude that in the nontrivial case ($\lambda$ not an eigenvalue of $\A$) the matrix
$
I_n + (\mathrm{diag}(\alpha_i)_{i=1}^n - \lambda I_n)^{-1}S^{-1} \delta\A  S
$
must be singular. Hence
$$
1 \leq \| (\mathrm{diag}(\alpha_i)_{i=1}^n - \lambda I_n)^{-1}S^{-1}\delta\A S \|_M \leq \| L^* (\mathrm{diag}(\alpha_i)_{i=1}^n - \lambda I_n)^{-1} L^{-*} \|_2 \| S^{-1}\delta\A S\|_M .
$$
Take any diagonal positive definite $\Delta$ and write $\Delta M \Delta =M_s$, $L_s = \Delta L$ (thus $M_s = L_s L_s^{*}$) and 
$$
1 \leq \|L_s\|_2 \|L_s^{-1}\|_2 \frac{1}{\min_{i=1:n}|\alpha_i-\lambda|} \|S^{-1}\delta\A S\|_M \leq \frac{\sqrt{\kappa_2(M_s)}}{\min_{i=1:n}|\alpha_i-\lambda|} 
\|S^{-1}\|_M \|S\|_M \|\delta\A \|_M . 
$$ 
Since this holds for an arbitrary $\Delta\in\mathrm{Diag}^+_n$, the condition number of $M$ enters the error bounds as $\mu_2(M)$, which is potentially much smaller that $\kappa_2(M)$. (For instance, if $M$ is diagonal with arbitrarily high $\kappa_2(M)$, the value of $\mu_2(M)$ is one.)
The second assertion follows from the first one by virtue of the  nice trick from \cite[\S 2]{eis-ipsen-1998}. 
\end{proof}

In the application to eigenvalue error estimates in DMD, the perturbation $\delta\A$ is the residual which should be measured using Proposition \ref{zd:PROP:RQ-weighted}.

Let us also mention that in the case of heavily weighted data (e.g. strongly graded row norms in $\X_m$) it is advised that the SVD and also the QR factorization (such as e.g. in Algorithm \ref{zd:ALG:DMD:RRR:compressed}) are computed more carefully, using row pivoting scheme, see \cite{pow-rei-68}, \cite{drm-ves-VW-1}, \cite{drm-ves-VW-2}.

\subsection{More general weighting schemes}
In Line 2 of Algorithm \ref{zd:ALG:POD}, $\X_m = \widehat{U} \Sigma V^*$ is an SVD of $\X_m$, where $\widehat{U}\equiv L^{-*}\widetilde{U}$ is $M$-unitary and $V$ is unitary (in the canonical Euclidean inner product). We can think of situations when one would want to weigh the snapshots throughout time steps as well. For instance, in a POD analysis of flames in the internal combustion engine \cite{POD-weighted-flames-2007}, the weighting is determined as a function of in-cylinder pressure and used to define weighted average for particular crank angle. As another example, we recall \S \ref{S=New-DMD-RRR}, where we used columns scaling of the data (snapshots) as  numerical device to curb ill-conditioning. 

To define a general setting, in addition to the $M$--induced structure discussed in \S \ref{SS=DMD-weighted-new-alg}, equip $\C^m$ with a weighted inner product $(\cdot,\cdot)_N$ generated by a positive definite matrix $N=KK^*$, or by its inverse $N^{-1}$. (For instance, $N$ could be diagonal with forgetting factors that put less weight on older snapshots, or to represent some other changes in time dynamics.) 

With the two norms induced by $M$ and $N$, most appropriate is the weighted SVD,  introduced by van Loan \cite{van-Loan-GSVD} and successfully deployed in \cite{ewe-luk-89} for solving various weighted least squares problems.
In this framework, Line 2 of Algorithm \ref{zd:ALG:POD} should be changed to computing the SVD 
\begin{equation}
L^* \X_m K^{-*} = \widetilde{U} \Sigma \widetilde{V}^*,\;\;\Sigma=\mathrm{diag}(\sigma_i)_{i=1}^m,\; \sigma_{1}\geq\cdots\geq\sigma_m ; \;\; \widetilde{U}^*\widetilde{U} = \widetilde{V}^*\widetilde{V}=I_m,
\end{equation}
which yields the weighted SVD $\X_m = \widehat{U} \Sigma \widehat{V}^{-1}$  with $M$-unitary $\widehat{U}=L^{-*}\widetilde{U}$ and $N$-unitary $\widehat{V}=K^{-*}\widetilde{V}$.
Recall that $\|\X_m\|_{N,M}=\max_{z\neq 0} \|\X_m z\|_M /\|z\|_N \equiv \|L^* \X_m K^{-*}\|_2=\sigma_1$. The $\|\cdot\|_{N,M}$ error of the best rank $k$ approximation $\widehat{U}_k\Sigma_k \widehat{V}_k^{-1}$ is $\sigma_{k+1}$. We can easily check that the Rayleigh quotient and the residual norm of a Ritz pair $(\lambda, \widehat{U}_k w)$ read, respectively, 
\begin{equation}\label{zd:eq:S_k-L*-K-*}
\widehat{S}_k = \widetilde{U}_k^* L^*\Y_m K^{-*}\widetilde{V}_k\Sigma_k^{-1},\;\;\| \A (\widehat{U}_k w) - \lambda (\widehat{U}_k w) \|_M = \| L^* \Y_m K^{-*} V_k\Sigma_k^{-1} w - \lambda \widetilde{U}_k w\|_2 .
\end{equation}
It is obvious how to adapt  the above formulas if the inner products are generated by $M^{-1}$ and $N^{-1}$. 
Also, the refinement procedure from \S \ref{zd:SS=refined-vectors} applies here as well,  with the difference that the norm is different and thus the minimal singular value of a product of three matrices is needed. The compressed DMD from \S \ref{zd:SS=Compressed-DMD} can be adapted to this general setting with necessary changes; e.g. the QR factorization (\ref{zd:eq:QRFm+1}) should be computed using $M$-unitary $Q_f$.

In all cases discussed above, we need the SVD of the product of three matrices that can be computed as in \cite{ewe-luk-89}, \cite{boj-etal-91}, and \cite{Drmac-triplet-SVD}. 

\section{Concluding remarks}
We have presented modifications of the DMD algorithm, together with theoretical analysis and justification, discussion of the potential weaknesses of the original method, and examples that illustrate the advantages of the new proposed method. From the point of view of numerical linear algebra, the deployed techniques are not new; however, the novelty is in adapting them to the data driven setting and turning the DMD into a more powerful tool. Also, we provide the fine and sometimes mission critical details of numerical software development.

  Moreover, we have defined a more general framework for the DMD, allowing formulations in weighted spaces, i.e. in elliptic norms defined by positive definite matrices. We hope that these modifications and new ideas will trigger further development of data driven methods for spectral analysis. 

Given the importance of DMD in computational analysis of various phenomena in applied sciences and engineering, we believe that our work will also facilitate more advanced and robust computational studies of dynamical systems, in particular in computational fluid mechanics.
For first encouraging applications of the new method see \cite{Crnjaric-zic:201708.002}, \cite{Crnkovic:201708.001}.  
\section*{Acknowledgments}
This  research  is supported  by  the  DARPA  Contract  HR0011-16-C-0116 \emph{``On a Data-Driven, Operator-Theoretic Framework for Space-Time Analysis of Process Dynamics''} and the ARL Contract W911NF-16-2-0160 \emph{``Numerical Algorithms for Koopman Mode Decomposition and Application in Channel Flow''}.

%

%

\printbibliography

\end{document}